\documentclass[10pt,twoside,reqno]{amsart}

\usepackage[margin=1in]{geometry}

\usepackage{mathtools}
\usepackage{booktabs}
\usepackage{upgreek}
\usepackage{float}
\usepackage[margin=1.5cm]{caption}

\newtheorem{theorem}{Theorem}[section]
\newtheorem{definition}[theorem]{Definition}
\newtheorem{corollary}[theorem]{Corollary}
\newtheorem{lemma}[theorem]{Lemma}

\numberwithin{equation}{section}
\newcommand\numberthis{\addtocounter{equation}{1}\tag{\theequation}}

\begin{document}

\title[The complex zeros of random orthogonal polynomials]{The complex zeros of random orthogonal polynomials}

\author[Christopher Corley]{Christopher Corley}

\address{Department of Mathematics, The University of Tennessee at Chattanooga, 304 Lupton Hall (Dept. 6956), 615 McCallie Avenue, Chattanooga, Tennessee 37403, United States of America}

\email{hry385@mocs.utc.edu}

\author[Andrew Ledoan]{Andrew Ledoan$^{\ast}$}

\thanks{$^{\ast}$The research of the second author was supported by National Science Foundation Grant DMS-1852288.}

\address{Department of Mathematics, The University of Tennessee at Chattanooga, 317 Lupton Hall (Dept. 6956), 615 McCallie Avenue, Chattanooga, Tennessee 37403, United States of America}

\email{andrew-ledoan@utc.edu}

\author[Aaron Yeager]{Aaron Yeager}

\address{Department of Mathematics, College of Coastal Georgia, University System of Georgia, 142 Jones Building, One College Drive, Brunswick, Georgia 31520, United States of America}

\email{ayeager@ccga.edu}

\makeatletter
\@namedef{subjclassname@2020}{\textup{2020} Mathematics Subject Classification}
\makeatother



\begin{abstract}
We utilize Cauchy's argument principle in combination with the Jacobian of a holomorphic function in several complex variables and the first moment of a ratio of two correlated complex normal random variables to prove explicit formulas for the density and the mean distribution of complex zeros of random polynomials spanned by orthogonal polynomials on the unit circle and on the unit disk. We then inquire into the consequences of their asymptotical evaluations.
\end{abstract}

\maketitle

\thispagestyle{empty}

\section{Introduction} \label{section_1}

Random polynomials are deeply studied and discussed. The literature on this beautiful and fascinating subject is extensive. Of the works which have come under our observation we mention the papers \cite{SheppVanderbei1995} by Shepp and Vanderbei and \cite{IbragimovZeitouni1997} by Ibragimov and Zeitouni. More specifically, Shepp and Vanderbei studied the complex zeros of the random algebraic polynomial $H_{n} (z) = \sum_{j = 0}^{n} \epsilon_{j} z^{j}$ with $z \in \mathbb{C}$ whose coefficients $\epsilon_{j}$ are independent and identically distributed standard real normal random variables. Combined with the classical argument principle, the Cholesky factorization generates four uncorrelated standard real normal and hence independent random variables from the real and imaginary parts of $H_{n}$ and $H_{n}^{\prime}$. These random variables are used to express the mean of $H_{n}^{\prime} / H_{n}$ as a complex-valued function, which is evaluated using the calculus of residues. Their method produces striking explicit formulas that generalize Kac's \cite{Kac1943} density and integral formula for the mean number of zeros in any measurable subset of the complex plane. The most characteristic and indeed the defining thing about the zeros is that they tend to concentrate asymptotically near the unit circle $\mathbb{T}$. By a different method, Ibragimov and Zeitouni studied the properties of the zeros under a wide class of distributions of the coefficients. Furthermore, they corroborated Shepp and Vanderbei's results for the normal distribution. The present paper is motivated by the results of papers \cite{SheppVanderbei1995, IbragimovZeitouni1997}. One of the objects we aim at is to extend the explicit formulas and asymptotic expansions for the density and the mean number of zeros to random polynomials spanned by orthogonal polynomials on $\mathbb{T}$ (OPUC) of the Nevai class, Szeg\H{o} class, and Sthal, Totik, and Ullman class, and by orthogonal polynomials on the unit disk $\mathbb{D}$, also known as Bergman polynomials.

We begin by establishing some notation. Let $\{\alpha_{j}\}_{j = 0}^{n}$ and $\{\beta_{j}\}_{j = 0}^{n}$ be sequences of independent and identically distributed standard real normal random variables defined on a probability space determined by the 3-tuple $(\Omega, \mathfrak{F}, \operatorname{Pr})$. As per usual, $\Omega$ is a set with generic elements $\omega$, $\mathfrak{F}$ is a $\sigma$-field of subsets of $\Omega$, and $\operatorname{Pr} \colon \mathfrak{F} \rightarrow [0, 1]$ is a probability measure on $\mathfrak{F}$. Let $\{\eta_{j}\}_{j = 0}^{n}$ be a sequence of standard complex normal random variables $\eta_{j} = \alpha_{j} + i \beta_{j}$ with density $\frac{1}{\pi} e^{z \overline{z}}$, mean $\operatorname{\mathbb{E}} [\eta_{j}] = 0$, and variance $\operatorname{\mathbb{E}} [\eta_{j} \overline{\eta_{j}}] = 1$. Let $\{\phi_{j}\}_{j = 0}^{n}$ be a sequence of polynomials satisfying the conditions $\deg \phi_{j} = j$ and $\phi_{0} = 1$. Then, let $\Phi_{n} (z) = \sum_{j = 0}^{n} \eta_{j} \phi_{j} (z)$ denote the random polynomial of order $n$ defined by the sequences $\{\eta_{j}\}_{j = 0}^{n}$ and $\{\phi_{j}\}_{j = 0}^{n}$. Further, let $N_{\Theta}^{\Phi_{n}, v}$ denote the random number of complex roots in an arbitrary domain $\Theta \subset \mathbb{C}$ of the random equation $\Phi_{n} (z) = v$ with $v \in \mathbb{C}$.

Our starting point is the argument principle. The method requires several new ingredients. In particular, the proof of the explicit formula for the mean $\operatorname{\mathbb{E}} [N_{\Theta}^{\Phi_{n}, v}]$ is based upon the explicit role of the Jacobian in the pairing of the roots in $\Theta$ of the random equation $\Phi_{n} (z) = v$ with the coefficients of the random polynomial $\Phi_{n} - v$ combined with Wu's \cite{Wu2019} closed form solution to the first moment of a ratio of two complex normal random variables. As will be shown later, the computation of a holomorphic Jacobian matrix of this mapping allows for the interchange of mean and contour integration. Furthermore, the numerator and denominator of the ratio random variable are correlated and have arbitrary means. Their real and imaginary parts and covariances must satisfy certain constraints. This argument principle and random ratio-based method eliminates the requirement of having to treat these real and imaginary parts via orthogonal factorization and to compute the residues of the mean $\operatorname{\mathbb{E}} [\Phi_{n}^{\prime} / (\Phi_{n} - v)]$. Wu's theorem can be stated as follows:

\begin{theorem}[Wu \protect{\cite{Wu2019}, Theorem 1}] \label{Thm-1.1}
Let $\mathbf{X} = \begin{bmatrix*} X_{1} & X_{2} \end{bmatrix*}^{T}$ be a general complex-valued bivariate normal random variable vector with mean
\begin{equation*}
\operatorname{\mathbb{E}} [\mathbf{X}]
 =
\begin{bmatrix*}
\mu_{X_{1}} \\ \addlinespace[1mm]
\mu_{X_{2}}
\end{bmatrix*}
\end{equation*}
and covariance matrix
\begin{equation*}
\operatorname{\mathbb{E}}
\begin{bmatrix*}
(X_{1} - \mu_{X_{1}}) (\overline{X_{1} - \mu_{X_{1}}}) & (X_{1} - \mu_{X_{1}}) (\overline{X_{2} - \mu_{X_{2}}}) \\ \addlinespace[1mm]
(X_{2} - \mu_{X_{2}}) (\overline{X_{1} - \mu_{X_{1}}}) & (X_{2} - \mu_{X_{2}}) (\overline{X_{2} - \mu_{X_{2}}})
\end{bmatrix*}
 = 
\begin{bmatrix*}
\sigma_{X_{1}}^{2} & \rho \sigma_{X_{1}} \sigma_{X_{2}} \\ \addlinespace[1mm]
\overline{\rho} \sigma_{X_{2}} \sigma_{X_{1}} & \sigma_{X_{2}}^{2}
\end{bmatrix*},
\end{equation*}
where
\begin{equation*}
\sigma_{X_{j}}^{2}
 = \operatorname{\mathbb{E}} [\lvert X_{j} \rvert^{2}] + \lvert \operatorname{\mathbb{E}} [X_{j}] \rvert^{2} \quad \text{with $\sigma_{X_{j}}^{2} \in \mathbb{R}_{\geq 0}$}
\end{equation*}
and
\begin{equation*}
\sigma_{X_{j}}
 = \sqrt{\sigma_{X_{j}}^{2}} \quad \text{for $j = 1, 2$};
\end{equation*}
furthermore,
\begin{equation*}
\rho
 = \operatorname{\mathbb{E}} \left[\frac{X_{1} \overline{X_{2}}}{\sigma_{X_{1}} \sigma_{X_{2}}}\right] \quad \text{with $\rho \in \mathbb{C}$ and $\, \lvert \rho \rvert \leq 1$}.
\end{equation*}
If $X_{2}$ has mean $\mu_{X_{2}} = 0$, then the ratio random variable $X_{1} / X_{2}$ has mean
\begin{equation} \label{eq-1.1}
\operatorname{\mathbb{E}} \left[\frac{X_{1}}{X_{2}}\right]
 = \frac{\rho \sigma_{X_{1}}}{\sigma_{X_{2}}}.
\end{equation}
Contrarily, if $\mu_{X_{2}} \neq 0$, then
\begin{equation} \label{eq-1.2}
\operatorname{\mathbb{E}} \left[\frac{X_{1}}{X_{2}}\right]
 = \frac{\mu_{X_{1}}}{\mu_{X_{2}}} + \left(\frac{\rho \sigma_{X_{1}}}{\sigma_{X_{2}}} - \frac{\mu_{X_{1}}}{\mu_{X_{2}}}\right) \exp \left(-\frac{\lvert \mu_{X_{2}} \rvert^{2}}{\sigma_{X_{2}}^{2}}\right).
\end{equation}
\end{theorem}

Theorem \ref{Thm-1.1} suggests that $\operatorname{\mathbb{E}} [X_{1} / X_{2}] \neq \mu_{X_{1}} / \mu_{X_{2}}$ in general when $\mu_{X_{1}} \neq 0$. Two immediate corollaries are: (a) The expression \eqref{eq-1.2} approaches the expression \eqref{eq-1.1} when $\lvert \mu_{X_{2}} \vert \to 0$; (b) in the case where $X_{1}$ and $X_{2}$ are independent, the expression \eqref{eq-1.2} approaches 0 when $\lvert \mu_{X_{2}} \vert \to 0$. In our application we need to consider the covariance and pseudo-covariance matrices and to show that the latter matrix vanishes.

To state our main new theorems we consider the reproducing kernel $K_{n} (z, w) = \sum_{j = 0}^{n} \phi_{j} (z) \overline{\phi_{j} (w)}$, where $z$ and $w$ are points of $\mathbb{C}$, and, for nonnegative integers $k$ and $l$, its derivatives $K_{n}^{(k, l)} (z, w) = \sum_{j = 0}^{n} \phi_{j}^{(k)} (z) \overline{\phi_{j}^{(l)} (w)}$.

\begin{theorem} \label{Thm-1.2}
For each domain $\Theta \subset \mathbb{C}$,
\begin{equation*}
\operatorname{\mathbb{E}} [N_{\Theta}^{\Phi_{n}, v}]
 = \frac{1}{2 \pi i} \oint_{\partial \Theta} \frac{\overline{K_{n}^{(0, 1)} (z, z)}}{K_{n} (z, z)} \, \exp \left(-\frac{\lvert v \rvert^{2}}{K_{n} (z, z)}\right) dz.
\end{equation*}
\end{theorem}

This relationship between $N_{\Theta}^{\Phi_{n}, v}$, $K_{n}$, and $K_{n}^{(k, l)}$ can be stated more explicitly as a real-valued double integral, which shows the precise dependency of $N_{\Theta}^{\Phi_{n}, v}$ with respect to $v$.

\begin{theorem} \label{Thm-1.3}
For each domain $\Theta \subset \mathbb{C}$,
\begin{equation*}
\operatorname{\mathbb{E}} [N_{\Theta}^{\Phi_{n}, v}]
 = \iint_{\Theta} \rho_{n, v} (z) \, dx \, dy,
\end{equation*}
where
\begin{equation*}
\rho_{n, v} (z)
 = \frac{1}{\pi} \left[\frac{K_{n}^{(1, 1)} (z, z)}{K_{n} (z, z)} - \frac{\lvert K_{n}^{(0, 1)} (z, z) \rvert^{2}}{K_{n} (z, z)^{2}} \left(1 - \frac{\lvert v \rvert^{2}}{K_{n} (z, z)}\right)\right] \exp \left(-\frac{\lvert v \rvert^{2}}{K_{n} (z, z)}\right).
\end{equation*}
\end{theorem}

The function $\rho_{n, v}$ represents the mean number of complex roots of $\Phi_{n} (z) = v$ per unit area at the point $(x, y) \in \mathbb{R}^{2}$. Integrating $\rho_{n, v}$ over any domain $\Theta$ produces the mean number of complex roots in $\Theta$, the effect of which is that the ratio $\rho_{n, v} / \iint_{\Theta} \rho_{n, v} \, dx \, dy$ then is the probability density of a complex root. The formula for $\rho_{n, v}$ is reminiscent of the formula for the density of Theorem 1 of \cite{CorleyLedoan2020}. The latter was proved using the techniques for conditional random fields developed in \cite{AdlerTaylor2007, AzaisWschebor2009, IbragimovZeitouni1997}. It generalizes Edelman and Kostlan's formula (see Theorem 3.1 of \cite{EdelmanKostlan1995}) for the density of zeros of a random analytic function.

The following statements are straightforward:

\begin{corollary} \label{Cor-1.4}
The following limits hold: $\operatorname{\mathbb{E}} [N_{\Theta}^{\Phi_{n}, v}]$ and $\rho_{n, v}$ each approaches 0 when $\lvert v \rvert \to \infty$.
\end{corollary}

\begin{corollary} \label{Cor-1.5}
For each domain $\Theta \subset \mathbb{C}$,
\begin{equation*}
\operatorname{\mathbb{E}} [N_{\Theta}^{\Phi_{n}, 0}]
 = \frac{1}{2 \pi i} \oint_{\partial \Theta} \frac{\overline{K_{n}^{(0, 1)} (z, z)}}{K_{n} (z, z)} \, dz
 = \iint_{\Theta} \rho_{n, 0} (z) \, dx \, dy,
\end{equation*}
where
\begin{equation*}
\rho_{n, 0} (z)
 = \frac{K_{n} (z, z) K_{n}^{(1, 1)} (z, z) - \lvert K_{n}^{(0, 1)} (z, z) \rvert^{2}}{\pi K_{n} (z, z)^{2}}.
\end{equation*}
\end{corollary}

It is worth noting that when $\phi_{j} (z) = z^{j}$ the density $\rho_{n, 0}$ becomes the complex analog of Kac's \cite{Kac1943} density. A more compact form was derived by Arnold \cite{Arnold1966} (see also Theorem 4.15 of \cite{Bharucha-ReidSambandham1986}). His result is a specialization of Theorem 1.1, part (ii), of \cite{Yeager2018} in the case where the $\phi_{j}$ are OPUC.

We mention that Rice's formula \cite{Rice1945, Rice1958} counts the mean number of crossings that a continuous-time stationary ergodic process per unit time crosses a fixed real level. For the special case of a polynomial Farahmand and Jahangiri \cite{FarahmandJahangiri1998} scrutinized the complex level crossings of the random algebraic equation $\sum_{j = 0}^{n} \eta_{j} g_{j} z^{j} = v$. The inclusion of the real constants $g_{j}$ allows for the coefficients to be nonidentically distributed. Therefore, if $N_{\Theta}^{\Phi_{n}, v}$ is interpreted as the number of complex $v$-level crossings in $\Theta$ of $\Phi_{n}$, then $\rho_{n, v}$ represents the mean number of crossings about the level $v$ per unit area at the point $(x, y) \in \mathbb{R}^{2}$.

The remainder of this paper is organized into eight sections. In Section \ref{section_2} we gather our principal tools that we require to be able to give a complete proof of Theorem \ref{Thm-1.2}. In Section \ref{section_3} we prove Theorem \ref{Thm-1.2}. In Section \ref{section_4} we prove Theorem \ref{Thm-1.3}. In Sections \ref{section_5} through \ref{section_8} we give some applications related to orthogonal polynomials. More precisely, we confine our attention to three different classes of OPUC for the basis functions: Nevai, Szeg\H{o}, and Sthal, Totik, and Ullman. The first question we consider is: What is the asymptotic formula for $\rho_{n, v}$? The next question is: What is the behavior of $\operatorname{\mathbb{E}} [N_{\Theta}^{\Phi_{n}, v}]$ in $\Theta \subseteq \mathbb{D}$ as $n \to \infty$? For some purposes it is convenient to establish analogous results with Bergman polynomials. In this connection we mention that in \cite{DiaconisEvans2001} the random polynomials spanned by Bergman polynomials on $\mathbb{D}$ are logarithmic derivatives of characteristic polynomials of random unitary or orthogonal matrices. In \cite{StahlTotik1992} Bergman polynomials are used as the basis functions for representing analytic functions on $\mathbb{D}$. Finally, in Section \ref{section_9} we give a numerical simulation of the zeros and $v$-level crossings based on the formulas for $\operatorname{\mathbb{E}} [N_{\Theta}^{\Phi_{n}, v}]$ and $\rho_{n, v}$ for a random polynomial $\Phi_{n}$ spanned by Bergman polynomials.

\section{Lemmata} \label{section_2}

A consequence of the fundamental theorem of algebra is that $P_{n} (z) = \sum_{j = 0}^{n} a_{j} z^{j} = a_{n} \prod_{j = 1}^{n} (z - z_{j})$ with $a_{j}, z_{j} \in \mathbb{C}$ and $a_{n} \neq 0$. Assume that $P_{n}$ is monic. Comparison of the representations shows that the dependence between the zeros and coefficients can be obtained in the form of Vi\`{e}te's formul$\ae$ (1.2.2) of \cite{RahmanSchmeisser2002} but with $a_{n} = 1$:
\begin{equation} \label{eq-2.1}
\arraycolsep=0pt\def\arraystretch{1}
\begin{array}{lcl}
z_{1} + z_{2} + z_{3} + \dotsb + z_{n} & = & - a_{n - 1}, \\ \addlinespace[1mm]
z_{1} z_{2} + z_{1} z_{3} + z_{2} z_{3} + \dotsb + z_{n - 1} z_{n} & = & \hspace{0.78em} a_{n - 2}, \\ \addlinespace[1mm]
z_{1} z_{2} z_{3} + z_{1} z_{2} z_{4} + z_{1} z_{3} z_{4} + \dotsb + z_{n - 2} z_{n - 1} z_{n} & = & - a_{n - 3}, \\ \addlinespace[1mm]
& \vdotswithin{=} & \\ \addlinespace[1mm]
z_{1} z_{2} z_{3} \dotsm z_{n} & = & (-1)^{n} a_{0}.
\end{array}
\end{equation}

\begin{definition}[Krantz \cite{Krantz1992}, Definition 1.4.9] \label{Def-2.1}
Let $\Lambda \subset \mathbb{C}^{n}$ be an open set, not necessarily connected, such that $\partial \Lambda = \partial (\mathbb{C}^{n} \backslash \overline{\Lambda})$. Let $\boldsymbol{f} = (f_{1}, \dotsc, f_{n}) \colon \Lambda \rightarrow \mathbb{C}^{n}$ be a holomorphic mapping. Write $w_{j} = f_{j} (z)$ for $j = 1, \dotsc, n$ to be the points $w_{j} \in \mathbb{C}^{n}$ assigned by $f_{j}$ to $z \in \Lambda$. Then the holomorphic Jacobian matrix of $\boldsymbol{f}$ is the matrix $J_{\mathbb{C}} \boldsymbol{f} = \partial (w_{1}, \dotsc, w_{n}) / \partial (z_{1}, \dotsc, z_{n})$. Write further $z_{j} = x_{j} + i y_{j}$ and $w_{k} = \beta_{k} + i \gamma_{k}$ for $j, k = 1, \dotsc, n$. Then the real Jacobian matrix of $\boldsymbol{f}$ is the matrix $J_{\mathbb{R}} \boldsymbol{f} = \partial (\beta_{1}, \gamma_{1}, \dotsc, \beta_{n}, \gamma_{n}) / \partial (x_{1}, y_{1}, \dotsc, x_{n}, y_{n})$.
\end{definition}

\begin{lemma}[Krantz \protect{\cite{Krantz1992}, Proposition 1.4.10}] \label{Lemma-2.2}
With notation as in Definition \ref{Def-2.1}, if $\boldsymbol{f}$ is a holomorphic mapping, then $\det J_{\mathbb{R}} \boldsymbol{f}  = \lvert \det J_{\mathbb{C}} \boldsymbol{f} \rvert^{2}$.
\end{lemma}

If $\boldsymbol{f}$ is equidimensional, then, clearly, the matrices $J_{\mathbb{C}} \boldsymbol{f}$ and $J_{\mathbb{R}} \boldsymbol{f}$ are square. Using the present terminology, we have:

\begin{lemma} \label{Lemma-2.3}
For any $n$-tuple $z_{1}, \dotsc, z_{n}$ of arguments for the mapping $\psi \colon \mathbb{C}^{n} \rightarrow \mathbb{C}^{n}$ defined by \eqref{eq-2.1}, if $\psi (z_{1}, \dotsc, z_{n}) = (a_{n - 1}, \dotsc, a_{0})$, then $\det J_{\mathbb{R}} \psi (z_{1}, \dotsc, z_{n}) = \prod_{1 \leq j < k \leq n} \lvert z_{j} - z_{k} \rvert^{2}$.
\end{lemma}

\begin{proof}
The mapping $\psi$ is holomorphic on $\mathbb{C}^{n}$. The equations \eqref{eq-2.1} lead to the computation
\begin{equation} \label{eq-2.2}
\det J_{\mathbb{C}} \psi (z_{1}, \dotsc, z_{n})
 = (-1)^{n + 1} \det \mathbf{B},
\end{equation}
where
\begin{equation*}
\mathbf{B}
 = 
\begin{bmatrix*}
1 & 1 & \dotso & 1 \\ \addlinespace[1mm]
\displaystyle \sum_{j = 2}^{n} z_{j} & \displaystyle \sum_{\substack{j = 1 \\ j \neq 2}}^{n} z_{j} & \dotso & \displaystyle \sum_{j = 1}^{n - 1} z_{j} \\ \addlinespace[1mm]
\vdots & \vdots & \ddots & \vdots \\ \addlinespace[1mm]
\displaystyle \prod_{j = 2}^{n} z_{j} & \displaystyle \prod_{\substack{j = 1 \\ j \neq 2}}^{n} z_{j} & \dotso & \displaystyle \prod_{j = 1}^{n - 1} z_{j}
\end{bmatrix*}.
\end{equation*}
A curious aspect of the matrix $\mathbf{B}$ is that two of the columns have equal elements if and only if the two zeros $z_{j}$ and $z_{k}$ with $j \neq k$ are equal, and so its determinant equals zero. Hence the determinant of the holomorphic Jacobian matrix of $\psi$ contains the Vandermonde determinant of $z_{1}, \dotsc, z_{n}$ as a factor, namely, the product of all differences of the form $z_{j} - z_{k}$ for $j = 1, \dotsc, k$ and $k = 1, \dotsc, n$.

If we compute $\det \mathbf{B}$ by expanding along the first row, we get summands of the form that contains powers of products of $z_{1}, \dotsc, z_{n}$ which sum to $\frac{1}{2} n (n - 1)$. In this manner $\det \mathbf{B}$ and, accordingly, $\det J_{\mathbb{C}} \psi$ are homogeneous polynomials in the variables $z_{1}, \dotsc, z_{n}$ of degree $\frac{1}{2} n (n - 1)$. When we carry out the multiplication on $\prod_{1 \leq j < k \leq n} (z_{j} - z_{k})$, it becomes apparent that this product is a homogeneous polynomial in $z_{1}, \dotsc, z_{n}$ of the same degree. For summands of the form $z_{1} \dotsm z_{n}$ cancel off. These polynomials in fact are equal. It suffices, therefore, to show that they have all their coefficients in common.

Denote $\mathbf{B} = [b_{j k}]_{j, k = 1}^{n}$. By equation (8) in Section 2.1 of \cite{LancasterTismenetsky1985}, we see that $\det \mathbf{B}$ is the sum of $n!$ summands, one for each permutation $\pi$ in the symmetric group of degree $n$. Each summand equals $(-1)^{N (\pi)} b_{1 \pi (1)} \dotsm b_{n \pi (n)}$, where $N (\pi)$ is the number of inversions in $\pi$. One and only one of the elements $b_{1 \pi (1)}, \dotsc, b_{n \pi (n)}$ is taken from each row $j$ and each column $\pi (j)$. Then $\det \mathbf{B}$ and $\prod_{1 \leq j < k \leq n} (z_{j} - z_{k})$ have equal coefficients $(-1)^{k}$ for $k = 0, \dotsc, n$, of corresponding powers of $z_{1}, \dotsc, z_{n}$. Therefore, using \eqref{eq-2.2}, we obtain $\det J_{\mathbb{C}} \psi (z_{1}, \dotsc, z_{n}) = (-1)^{n + 1} \prod_{1 \leq j < k \leq n} (z_{j} - z_{k})$. An appeal to Lemma \ref{Lemma-2.2} completes the proof.
\end{proof}

Now we consider the general case of $Q_{n} (z) = \sum_{j = 0}^{n} \eta_{j} z^{j} - v$ with $\eta_{j} \in \mathbb{C}$ and zeros $\zeta_{1}, \dotsc, \zeta_{n}$. We apply Vi\`{e}te's formul$\ae$ and obtain
\begin{equation} \label{eq-2.3}
\arraycolsep=0pt\def\arraystretch{1}
\begin{array}{lcl}
\zeta_{1} + \zeta_{2} + \zeta_{3} + \dotsb + \zeta_{n} & = & \displaystyle - \frac{\eta_{n - 1}}{\eta_{n}}, \\ \addlinespace[1mm]
\zeta_{1} \zeta_{2} + \zeta_{1}\zeta_{3} + \zeta_{2}\zeta_{3} + \dotsb + \zeta_{n - 1} \zeta_{n} & = & \displaystyle \hspace{0.78em} \frac{\eta_{n - 2}}{\eta_{n}}, \\ \addlinespace[1mm]
\zeta_{1} \zeta_{2} \zeta_{3} + \zeta_{1} \zeta_{2} \zeta_{4} + \zeta_{1} \zeta_{3} \zeta_{4} + \dotsb + \zeta_{n - 2} \zeta_{n - 1} \zeta_{n} & = & \displaystyle - \frac{\eta_{n - 3}}{\eta_{n}}, \\ \addlinespace[1mm]
& \vdotswithin{=} & \\ \addlinespace[1mm]
\zeta_{1} \zeta_{2} \zeta_{3} \dotsm \zeta_{n} & = & \displaystyle (-1)^{n} \frac{\eta_{0} - v}{\eta_{n}}.
\end{array}
\end{equation}

\begin{lemma} \label{Lemma-2.4}
For any $n$-tuple $\zeta_{1}, \dotsc, \zeta_{n}$ of arguments for the mapping $\xi \colon \mathbb{C}^{n} \rightarrow \mathbb{C}^{n}$ defined by \eqref{eq-2.3}, if $\xi (\zeta_{1}, \dotsc, \zeta_{n}) = (\eta_{n - 1} / \eta_{n}, \dotsc, (\eta_{0} - v) / \eta_{n})$, then $\det J_{\mathbb{R}} \xi (\zeta_{1}, \dotsc, \zeta_{n}) = \lvert \eta_{n} \rvert^{2 n} \prod_{1 \leq j < k \leq n} \lvert \zeta_{j} - \zeta_{k} \rvert^{2}$.
\end{lemma}

\begin{proof}
The assertion is a direct consequence of Lemmata \ref{Lemma-2.2} and \ref{Lemma-2.3}, and an argument already used in the proof of Lemma \ref{Lemma-2.3} with \eqref{eq-2.2} replaced by $\det J_{\mathbb{C}} \xi (\zeta_{1}, \dotsc, \zeta_{n}) = (-1)^{n + 1} \eta_{n}^{n} \det \mathbf{B}$.
\end{proof}

\section{Proof of Theorem \ref{Thm-1.2}} \label{section_3}

Before embarking in the proof of Theorem \ref{Thm-1.2}, we mention that Shepp and Vanderbei's \cite{SheppVanderbei1995} explicit computation of the density and the mean number of zeros in $\Theta$ of $H_{n}$ was carried out without the necessary justification of the interchange of mean and contour integration, which they deemed to be ``tedious but doable.'' We do not think there is any indication anywhere of a solution to this intermediate problem. The present paper fills this gap. Now we prove Theorem \ref{Thm-1.2}.

\begin{proof}
Assume for simplicity that $\Phi_{n} - v$ has no zeros on $\Theta$. Since it is holomorphic within and on $\Theta$, by the argument principle (see Theorem 10.43 of \cite{Rudin1987}), we have
\begin{equation*}
N_{\Theta}^{\Phi_{n}, v}
 = \frac{1}{2 \pi i} \oint_{\partial \Theta} \frac{\Phi_{n}^{\prime} (z)}{\Phi_{n} (z) - v} \, dz.
\end{equation*}
The contour integral is to be evaluated over $\partial \Theta$ in the positive sense. We take the mean of both sides of this equation, apply the Tonelli--Fubini theorem (see Theorem 8.8 of \cite{Rudin1987}), and consequently obtain
\begin{equation} \label{eq-3.1}
\operatorname{\mathbb{E}} [N_{\Theta}^{\Phi_{n}, v}]
 = \frac{1}{2 \pi i} \oint_{\partial \Theta} \Psi_{n, v} (z) \, dz,
\end{equation}
where $\Psi_{n, v}$ is the complex-valued function
\begin{equation} \label{eq-3.2}
\Psi_{n, v} (z)
 = \operatorname{\mathbb{E}} \left[\frac{\Phi_{n}^{\prime} (z)}{\Phi_{n} (z) - v}\right].
\end{equation}
We justify the interchange of mean and contour integration in the following way:

Let
$\boldsymbol{\eta} = \begin{bmatrix*} \eta_{0} - v & \eta_{1} & \dotso & \eta_{n} \end{bmatrix*}^{T}$. Its complex normal probability density function is
\begin{equation*}
f (\boldsymbol{\eta})
 = \frac{1}{\pi^{n + 1}} \exp \left(-\lvert \eta_{0} - v \rvert^{2} - \sum_{j = 1}^{n} \, \lvert \eta_{j} \rvert^{2}\right).
\end{equation*}
It suffices to prove that $\lvert \Phi_{n}^{\prime} / (\Phi_{n} - v) \rvert$ is measurable and
\begin{equation} \label{eq-3.3}
\int_{\mathbb{R}^{2 n + 2}} \int_{\partial D_{\varrho}} \left\lvert \frac{\Phi_{n}^{\prime} (z)}{\Phi_{n} (z) - v}\right \rvert \left\lvert dz\right \rvert f (\boldsymbol{\eta}) \, dV
 < \infty,
\end{equation}
where $D_{\varrho} = \{z \in \mathbb{C} \colon \lvert z \rvert < \varrho\}$ is the open disk with radius $\varrho \in \mathbb{R}_{> 0}$ centered at the origin. Even though we consider $\Theta \subset \mathbb{C}$, seeing that $\operatorname{\mathbb{E}} [N_{\Theta}^{\Phi_{n}, v}]$ is a local property, it thus suffices in general to prove \eqref{eq-3.3} when we substitute $D_{\varrho}$ for $\Theta$. The argument can be broken down into two stages.

As a first step we prove \eqref{eq-3.3} for $\Phi_{n} (z) = \sum_{j = 0}^{n} \eta_{j} z^{j}$ with $\eta_{j} \in \mathbb{C}$. Let the zeros be $\zeta_{1, \boldsymbol{\eta}}, \dotsc, \zeta_{n, \boldsymbol{\eta}}$. We have
\begin{equation*}
\frac{\Phi_{n}^{\prime} (z)}{\Phi_{n} (z) - v}
 = \sum_{j = 1}^{n} \frac{1}{z - \zeta_{j, \boldsymbol{\eta}}}.
\end{equation*}
Thus, if
\begin{equation} \label{eq-3.4}
\int_{\mathbb{R}^{2 n + 2}} \int_{\partial D_{\varrho}} \left\lvert \frac{1}{z - \zeta_{j, \boldsymbol{\eta}}}\right \rvert \left\lvert dz\right \rvert f (\boldsymbol{\eta}) \, dV
 < \infty, \quad j = 0, \dotsc, n,
\end{equation}
then, clearly, \eqref{eq-3.3} would follow. And, if we consider a particular set $B \subset \mathbb{C}^{k}$, namely an infinitesimal hypercube centered at $\mathbf{y}_{0}$ of dimension $2 l$ on each side, so that $B = \{\mathbf{y} \in \mathbb{C}^{k} \colon \lvert \mathrm{y}_{1} - \mathrm{y}_{0 1} \rvert \leq l, \lvert \mathrm{y}_{2} - \mathrm{y}_{0 2} \rvert \leq l, \dotsc, \lvert \mathrm{y}_{k} - \mathrm{y}_{0 k} \rvert \leq l\}$, and let $V_{B}$ signify the volume of the hypercube in $k$ dimensions, $(2 l)^{k}$, and $dV_{k}$ be the volume measure in $\mathbb{C}^{k}$, we can write
\begin{equation*}
\int_{\mathbb{C}^{n + 1}} \int_{\partial D_{\varrho}} \left\lvert \frac{1}{z - \zeta_{j, \boldsymbol{\eta}}}\right \rvert \lvert dz \rvert \, f (\boldsymbol{\eta}) \, dV_{n + 1}
 = \int_{\mathbb{C}} \int_{\mathbb{C}^{n}} \int_{\partial D_{\varrho}} \left\lvert \frac{1}{z - \zeta_{j, \boldsymbol{\eta}}}\right \rvert \lvert dz \rvert \, f (\eta_{0} - v, \eta_{1}, \dotsc, \eta_{n - 1}) \, dV_{n} \, f (\eta_{n}) \, d\eta_{n}.
\end{equation*}
Now single out and fix $\eta_{n}$. We have to distinguish two cases.

In the first place, we may suppose that an arbitrary zero $\zeta_{j, \boldsymbol{\eta}}, \, j \in \{1, \dotsc, n\}$, lies on $\partial D_{\varrho}$. In order that \eqref{eq-3.4} may hold, the set $\{\boldsymbol{\eta} \in \mathbb{C}^{n} \colon \zeta_{j, \boldsymbol{\eta}} \in \partial D_{\varrho}\}$ must be negligible. Indeed, by Lemma \ref{Lemma-2.4}, we do have
\begin{equation*}
\begin{split}
V_{n} (\{\boldsymbol{\eta} \in \mathbb{C}^{n} \colon \zeta_{j, \boldsymbol{\eta}} \in \partial D_{\varrho}\})
 & = \int_{\xi (\{\zeta_{j, \boldsymbol{\eta}} \in \partial D_{\varrho}\})} dV_{n} \\
 & \leq \int_{(\partial D_{\varrho})^{n}} \lvert \eta_{n} \rvert^{2 n} \prod_{1 \leq j < k \leq n} \lvert \zeta_{j, \boldsymbol{\eta}} - \zeta_{k, \boldsymbol{\eta}} \rvert^{2} \, dV_{n}
 = 0,
\end{split}
\end{equation*}
where $(\partial D_{\varrho})^{n}$ denotes the product of $n$ circles with radius $\varrho$ centered at the origin, and thus we must have $V_{n + 1} (\{\boldsymbol{\eta} \in \mathbb{C}^{n} \colon \zeta_{j, \boldsymbol{\eta}} \in \partial D_{\varrho}\}) = 0$; hence, in particular, $\lvert \Phi_{n}^{\prime} / (\Phi_{n} - v) \rvert$ is measurable.

Secondly, we suppose that there are no zeros on $\partial D_{\varrho}$. If $\lvert z - \zeta_{j, \boldsymbol{\eta}} \rvert \geq 1$ whenever $z \in \partial D_{\varrho}$, then by direct inspection we know that
\begin{equation*}
\int_{\mathbb{R}^{2 n + 2}} \int_{\partial D_{\varrho}} \left\lvert \frac{1}{z - \zeta_{j, \boldsymbol{\eta}}}\right \rvert \lvert dz \rvert \, f (\boldsymbol{\eta}) \, dV
 \leq 2 \pi \varrho.
\end{equation*}
It is important to consider the case for which $0 < \lvert z - \zeta_{j, \boldsymbol{\eta}} \rvert < 1$ whenever $z \in \partial D_{\varrho}$. In that case, if $\zeta_{j, \boldsymbol{\eta}} \in D_{\varrho}$ should lie on the open ray $]0, \varrho[$, then we know that the maximum is reached by $\lvert 1 / (z - \zeta_{j, \boldsymbol{\eta}}) \rvert$ when $\arg z$ lies between $-\frac{1}{2} \pi$ and $\frac{1}{2} \pi$, with both endpoints  exclusive, since the real part of $z$ is positive. We also know that $\lvert z - \zeta_{j, \boldsymbol{\eta}} \rvert \geq \varrho$ when $\arg z$ rotates from $\frac{1}{2} \pi$ to $-\frac{1}{2} \pi$. We find without serious difficulty that
\begin{equation*}
\begin{split}
\int_{\partial D_{\varrho}} \left\lvert \frac{1}{z - \zeta_{j, \boldsymbol{\eta}}}\right \rvert \lvert dz \rvert
 & \leq \int_{-\frac{1}{2} \pi}^{\frac{1}{2} \pi} \left\lvert \frac{1}{z - \zeta_{j, \boldsymbol{\eta}}}\right \rvert \lvert dz \rvert + \pi \\
 & = \left(\int_{0}^{\varrho - \zeta_{j, \boldsymbol{\eta}}} + \int_{\varrho - \zeta_{j, \boldsymbol{\eta}}}^{\frac{1}{2} \pi}\right) \frac{2 \varrho}{\sqrt{(\varrho \cos t - \zeta_{j, \boldsymbol{\eta}})^{2} + (\varrho \sin t)^{2}}} \, dt + \pi \\
 & \leq \int_{0}^{\varrho - \zeta_{j, \boldsymbol{\eta}}} \frac{2 \varrho}{\sqrt{\varrho^{2} - 2 \varrho \zeta_{j, \boldsymbol{\eta}} + \zeta_{j, \boldsymbol{\eta}}^{2}}} \, dt + \int_{\varrho - \zeta_{j, \boldsymbol{\eta}}}^{\frac{1}{2} \pi} \frac{2 \varrho}{\sqrt{\varrho^{2} \sin^{2} t}} \, dt + \pi \\
 & \leq 2 \varrho + \pi \log \left(\frac{1}{\varrho - \zeta_{j, \boldsymbol{\eta}}}\right) + \pi \log \frac{\pi}{2} + \pi.
\end{split}
\end{equation*}
We may recall that $0 \leq \frac{2}{\pi} t \leq \sin t$ whenever $t \in [0, \frac{1}{2} \pi]$. In the penultimate step an upper estimation for the integral over $[\varrho - \zeta_{j, \boldsymbol{\eta}}, \frac{1}{2} \pi]$ is obtained with the aid of this fact.

We have now to show that
\begin{equation*}
\int_{\mathbb{R}^{2 n + 2}} \log \left(\frac{1}{\varrho - \zeta_{j, \boldsymbol{\eta}}}\right) f (\boldsymbol{\eta}) \, dV
 < \infty.
\end{equation*}
More careful reflection tells us that we have the rationale for replacing $\varrho - \zeta_{j, \boldsymbol{\eta}}$ with $\lvert \varrho - \lvert \zeta_{j, \boldsymbol{\eta}} \rvert \rvert$ from the conditions imposed on the zeros. By Lemma \ref{Lemma-2.4}, written in full,
\begin{equation} \label{eq-3.5}
\begin{split}
& \int_{\mathbb{R}^{2 n + 2}} \log \, \left\lvert \frac{1}{\varrho - \lvert \zeta_{j, \boldsymbol{\eta}} \rvert}\right \rvert f (\boldsymbol{\eta}) \, dV \\
 & \quad = \int_{\mathbb{R}^{2}} \int_{\mathbb{R}^{2 n}} \log \left\lvert \frac{1}{\varrho - \lvert \zeta_{j, \boldsymbol{\eta}} \rvert} \right\rvert f [\xi (\eta_{0} - v, \eta_{1}, \dotsc, \eta_{n - 1})] \, \lvert \eta_{n} \rvert^{2 n} \prod_{1 \leq k < l \leq n} \lvert \zeta_{k, \boldsymbol{\eta}} - \zeta_{l, \boldsymbol{\eta}} \rvert^{2} \, dV_{n} \, f (\eta_{n}) \, dV_{1},
\end{split}
\end{equation}
where
\begin{align*}
f [\xi (\eta_{0} - v, \eta_{1}, \dotsc, \eta_{n - 1})]
 & = f \left(-\eta_{n}\sum_{k = 1}^{n} \zeta_{k, \boldsymbol{\eta}}, \eta_{n} \sum_{k < l} \zeta_{k, \boldsymbol{\eta}} \zeta_{l, \boldsymbol{\eta}}, \dotsc, (-1)^{n} \eta_{n} \prod_{k = 1}^{n} \zeta_{k, \boldsymbol{\eta}}\right) \\
 & = \frac{1}{\pi^{n}} \exp \left(-\left\lvert \eta_{n} \sum_{k = 1}^{n} \zeta_{k, \boldsymbol{\eta}}\right \rvert^{2} - \left\lvert \eta_{n} \sum_{k < l} \zeta_{k, \boldsymbol{\eta}} \zeta_{l, \boldsymbol{\eta}}\right \rvert^{2} - \dotsb - \left\lvert \eta_{n} \prod_{k = 1}^{n} \zeta_{k, \boldsymbol{\eta}}\right \rvert^{2}\right) \numberthis  \label{eq-3.6} \\
 & \equiv g (\zeta_{1, \boldsymbol{\eta}}, \dotsc, \zeta_{n, \boldsymbol{\eta}}).
\end{align*}
We bound the integral
\begin{equation*}
\begin{split}
 & \int_{\mathbb{R}^{2}} \log \left\lvert \frac{1}{\varrho - \lvert \zeta_{j, \boldsymbol{\eta}} \rvert}\right \rvert g (\zeta_{1, \boldsymbol{\eta}}, \dotsc, \zeta_{n, \boldsymbol{\eta}}) \, dV_{1}
\end{split}
\end{equation*}
as follows:

If $r \geq \varrho + 1$, then $\lvert z - \zeta_{j, \boldsymbol{\eta}} \rvert \geq 1$ whenever $z \in \partial D_{\varrho}$. This subcase has already been treated.

Differently, if $r < \varrho + 1$, then through direct calculation and application of polar coordinates, with the integration of the zero $\zeta_{j, \boldsymbol{\eta}}$ to be extended over the closed ray $[0, \varrho + 1]$, we are thus led to
\begin{equation*}
\begin{split}
& \int_{\{\zeta_{j, \boldsymbol{\eta}} \colon 0 < \lvert z - \zeta_{j, \boldsymbol{\eta}} \rvert < 1, \, z \in \partial D_{\varrho}\}} \log \left\lvert \frac{1}{\varrho - \lvert \zeta_{j, \boldsymbol{\eta}} \rvert}\right \rvert g (\zeta_{1, \boldsymbol{\eta}}, \dotsc, \zeta_{j, \boldsymbol{\eta}}, \dotsc, \zeta_{n, \boldsymbol{\eta}}) \, dV_{1} \\
 & \quad = \left(\int_{0}^{2 \pi} \int_{0}^{\varrho} + \int_{0}^{2 \pi} \int_{\varrho}^{\varrho + 1}\right) \log \left\lvert \frac{1}{\varrho - r}\right \rvert g (\zeta_{1, \boldsymbol{\eta}}, \dotsc, r e^{i \theta}, \dotsc, \zeta_{n, \boldsymbol{\eta}}) \, r \, dr \, d\theta \\
 & \quad \leq 2 \pi \varrho^{2} (1 - \log \varrho) \max_{\boldsymbol{\eta} \in \{\boldsymbol{\mu} \in \mathbb{C}^{n} \colon \zeta_{k, \boldsymbol{\mu}} \in \overline{D_{\varrho}}, \, k = 1, \dotsc, n\}} g (\zeta_{1, \boldsymbol{\eta}}, \dotsc, \zeta_{n, \boldsymbol{\eta}}) \\
 & \quad \quad + 2 \pi (\varrho + 1) \max_{\boldsymbol{\eta} \in \{\boldsymbol{\mu} \in \mathbb{C}^{n} \colon \zeta_{k, \boldsymbol{\mu}} \in \overline{D_{\varrho + 1}} \backslash D_{\varrho}, \, k = 1, \dotsc, n\}} g (\zeta_{1, \boldsymbol{\eta}}, \dotsc, \zeta_{n, \boldsymbol{\eta}}).
\end{split}
\end{equation*}
The point to note is that, owing to the definition \eqref{eq-3.6} of $g$, the above two maximum values decay exponentially at infinity when each $\lvert \zeta_{k, \boldsymbol{\eta}} \rvert \to \infty, \, k = 1, \dotsc, n$.

We return to \eqref{eq-3.5}. Using the estimate $\lvert \zeta_{j, \boldsymbol{\eta}} \rvert < \varrho + 1$, we thus have
\begin{equation*}
\prod_{\substack{k = 1 \\ k \neq j}}^{n} \lvert \zeta_{j, \boldsymbol{\eta}} - \zeta_{k, \boldsymbol{\eta}} \rvert^{2}
 < \prod_{\substack{k = 1 \\ k \neq j}}^{n} (\varrho + 1 + \lvert \zeta_{k, \boldsymbol{\eta}} \rvert)^{2}
\end{equation*}
and at one stroke two finite and convergent integrals:
\begin{equation*}
\begin{split}
 & 2 \pi \varrho^{2} (1 - \log \varrho) \int_{\mathbb{R}^{2}} \int_{\mathbb{R}^{2 n - 2}} \lvert \eta_{n} \rvert^{2 n} \prod_{\substack{1 \leq k < l \leq n \\ k, l \neq j}} \lvert \zeta_{k, \boldsymbol{\eta}} - \zeta_{l, \boldsymbol{\eta}} \rvert^{2} \prod_{\substack{k = 1 \\ k \neq j}}^{n} (\varrho + 1 + \lvert \zeta_{k, \boldsymbol{\eta}} \rvert)^{2} \\ & \quad \times \max_{\boldsymbol{\eta} \in \{\boldsymbol{\mu} \in \mathbb{C}^{n} \colon \zeta_{k, \boldsymbol{\mu}} \in \overline{D_{\varrho}}, \, k = 1, \dotsc, n\}} g (\zeta_{1, \boldsymbol{\eta}}, \dotsc, \zeta_{n, \boldsymbol{\eta}}) \, f (\eta_{n}) \, dV_{n - 1} \, dV_{1}
\end{split}
\end{equation*}
and
\begin{equation*}
\begin{split}
 & 2 \pi (\varrho + 1) \int_{\mathbb{R}^{2}} \int_{\mathbb{R}^{2 n - 2}} \lvert \eta_{n} \rvert^{2 n} \prod_{\substack{1 \leq k < l \leq n \\ k, l \neq j}} \lvert \zeta_{k, \boldsymbol{\eta}} - \zeta_{l, \boldsymbol{\eta}} \rvert^{2} \prod_{\substack{k = 1 \\ k \neq j}}^{n} (\varrho + 1 + \lvert \zeta_{k, \boldsymbol{\eta}} \rvert)^{2} \\ & \quad \times \max_{\boldsymbol{\eta} \in \{\boldsymbol{\mu} \in \mathbb{C}^{n} \colon \zeta_{k, \boldsymbol{\mu}} \in \overline{D_{\varrho + 1}} \backslash D_{\varrho}, \, k = 1, \dotsc, n\}} g (\zeta_{1, \boldsymbol{\eta}}, \dotsc, \zeta_{n, \boldsymbol{\eta}}) \, f (\eta_{n}) \, dV_{n - 1} \, dV_{1}.
\end{split}
\end{equation*}
Thus the condition \eqref{eq-3.4} is fulfilled.

We start over again, let
$\Phi_{n} (z) = \sum_{j = 0}^{n} \tau_{j} \phi_{j} (z)$, where the $\tau_{j}$ are independent and identically distributed standard complex normal random variables, $\phi_{j} (z) = \sum_{k = 0}^{j} a_{j, k} z^{k}$ with $a_{j, k} \in \mathbb{C}$ for $k = 0, \dotsc, j$ and $j = 1, \dotsc, n$, and $\phi_{0} = 1$, and immediately observe that
\begin{equation*}
\begin{split}
\Phi_{n} (z) - v
 & = \tau_{n}a_{n, n} z^{n} \\
 & \quad + (\tau_{n} a_{n, n - 1} + \tau_{n - 1} a_{n - 1, n - 1}) z^{n - 1} \\
 & \quad + (\tau_{n}a_{n, n - 2} + \tau_{n - 1} a_{n - 1, n - 2} + \tau_{n - 2} a_{n - 2, n - 2}) z^{n - 2} \\
 & \quad \vdotswithin{=} \\
 & \quad + \tau_{n} a_{n, 0} + \tau_{n - 1} a_{n - 1, 0} + \dotsb + \tau_{0} - v.
\end{split}
\end{equation*}
If we can introduce, therefore, the mapping $\phi$ defined by the relationships
\begin{equation*}
\arraycolsep=0pt\def\arraystretch{1}
\begin{array}{lcl}
\tau_{n} a_{n, n} & = & \eta_{n}, \\ \addlinespace[1mm]
\tau_{n} a_{n, n -1} + \tau_{n - 1} a_{n - 1, n - 1} & = & \eta_{n - 1}, \\ \addlinespace[1mm]
\tau_{n} a_{n, n - 2} + \tau_{n - 1} a_{n - 1, n - 2} + \tau_{n - 2} a_{n - 2, n - 2} & = & \eta_{n - 2}, \\ \addlinespace[1mm]
& \vdotswithin{=} & \\ \addlinespace[1mm]
\tau_{n} a_{n, 0} + \tau_{n - 1} a_{n - 1, 0} + \dotsb + \tau_{0}- v & = & \eta_{0},
\end{array}
\end{equation*}
say, and set
\begin{equation*}
\boldsymbol{\tau}
 = 
\begin{bmatrix*}
\tau_{n} a_{n, n} \\ \addlinespace[1mm]
\tau_{n} a_{n, n -1} + \tau_{n - 1} a_{n - 1, n - 1} \\ \addlinespace[1mm]
\vdots \\ \addlinespace[1mm]
\tau_{n} a_{n, 0} + \tau_{n - 1} a_{n - 1, 0} + \dotsb + \tau_{0} - v
\end{bmatrix*},
\end{equation*}
so that the expression $\phi (\boldsymbol{\tau}) = \begin{bmatrix*} \eta_{n} & \eta_{n - 1} & \dotso & \eta_{0} \end{bmatrix*}^{T} \equiv \boldsymbol{\widetilde{\eta}}$ satisfies
$J_{\mathbb{R}} \phi = \prod_{j = 1}^{n} \lvert a_{j, j} \rvert^{2}$, then we can proceed in complete analogy with the manner in which we argued above and apply Lemma \ref{Lemma-2.4} in order to show that
\begin{equation*}
\int_{\phi (\mathbb{R}^{2 n + 2})} \int_{\partial \Theta} \left\lvert \frac{\Phi_{n}^{\prime} (z)}{\Phi_{n} (z) - v}\right \rvert \lvert dz \rvert \, f (\boldsymbol{\tau}) \, dV
 = \prod_{j = 1}^{n} \lvert a_{j, j} \rvert^{2} \int_{\mathbb{R}^{2 n + 2}} \int_{\partial \Theta} \left\lvert \frac{\Phi_{n}^{\prime} (z)}{\Phi_{n} (z) - v}\right \rvert \lvert dz \rvert \, f (\boldsymbol{\widetilde{\eta}}) \, dV
 < \infty.
\end{equation*}
Accordingly we can conclude that $\lvert \Phi_{n}^{\prime} / (\Phi_{n} - v) \rvert$ is measurable.

Continuing the proof of the theorem we introduce a new two-dimensional complex-valued random variable vector $\mathbf{Z}$ by
$\mathbf{Z} = \begin{bmatrix*} \Phi_{n}^{\prime} (z) & \Phi_{n} (z) - v \end{bmatrix*}^{T}$. Observe that it has a nonsingular complex bivariate normal distribution completely described through the complex normal probability density function
\begin{equation*}
f (\mathbf{Z} ; \boldsymbol{\upmu}_{\mathbf{Z}}, \boldsymbol{\Sigma}_{\mathbf{Z} \mathbf{Z}})
 = \frac{1}{\pi^{2} \det \boldsymbol{\Sigma}_{\mathbf{Z} \mathbf{Z}}} \exp [-(\mathbf{Z} - \boldsymbol{\upmu}_{\mathbf{Z}})^{H} \boldsymbol{\Sigma}_{\mathbf{Z} \mathbf{Z}}^{-1} (\mathbf{Z} - \boldsymbol{\upmu}_{\mathbf{Z}})].
\end{equation*}
The parameters mean $\boldsymbol{\upmu}_{\mathbf{Z}}$ and covariance matrix $\boldsymbol{\Sigma}_{\mathbf{Z} \mathbf{Z}} = \operatorname{cov} [\mathbf{Z}, \mathbf{Z}]$ satisfy
\begin{equation*}
\boldsymbol{\upmu}_{\mathbf{Z}}
 = \operatorname{\mathbb{E}} [\mathbf{Z}]
 = 
\begin{bmatrix*}
 0 \\ \addlinespace[1mm]
 -v
\end{bmatrix*}
\end{equation*}
and
\begin{equation*}
\boldsymbol{\Sigma}_{\mathbf{Z} \mathbf{Z}}
 = \operatorname{\mathbb{E}} [(\mathbf{Z} - \boldsymbol{\upmu}_{\mathbf{Z}}) (\mathbf{Z} - \boldsymbol{\upmu}_{\mathbf{Z}})^{H}]
 = \operatorname{\mathbb{E}} 
\begin{bmatrix*}
\Phi_{n}^{\prime} (z) \overline{\Phi_{n}^{\prime} (z)} & \Phi_{n}^{\prime} (z) \overline{\Phi_{n} (z)} \\ \addlinespace[1mm]
\overline{\Phi_{n}^{\prime} (z)} \Phi_{n} (z) & \Phi_{n} (z) \overline{\Phi_{n} (z)}
\end{bmatrix*},
\end{equation*}
where $\mathbf{A}^{H} = \overline{\mathbf{A}^{T}}$ denotes the Hermitian conjugate for a complex matrix $\mathbf{A}$. We apply Theorem \ref{Thm-1.1} to $\Psi_{n, v}$ in the special case where $\mathbf{Z}$ has a $\mathcal{C} \mathcal{N} (\boldsymbol{\upmu}_{\mathbf{Z}}, \boldsymbol{\Sigma}_{\mathbf{Z} \mathbf{Z}})$ distribution. As regards the pseudo-covariance matrix $\boldsymbol{\Sigma}_{\mathbf{Z} \overline{\mathbf{Z}}} = \operatorname{cov} [\mathbf{Z}, \overline{\mathbf{Z}}]$, we show that it is zero. This owes to the fact that the equivalent representation of $\mathbf{Z}$ as the real-valued random variable vector
\begin{equation*}
\begin{bmatrix*}
\begin{bmatrix*}
\operatorname{Re} [\Phi_{n}^{\prime} (z)] & \operatorname{Re} [\Phi_{n} (z) - v]
\end{bmatrix*}^{T} &
\begin{bmatrix*}
\operatorname{Im} [\Phi_{n}^{\prime} (z)] & \operatorname{Im} [\Phi_{n} (z) - v]
\end{bmatrix*}^{T}
\end{bmatrix*}^{T}
\end{equation*}
fulfills the conditions
\begin{equation*}
\operatorname{cov} [\operatorname{Re} [\Phi_{n}^{\prime} (z)], \operatorname{Re} [\Phi_{n} (z) - v]]
 = \operatorname{cov} [\operatorname{Im} [\Phi_{n}^{\prime} (z)], \operatorname{Im} [\Phi_{n} (z) - v]]
\end{equation*}
and
\begin{equation*}
\operatorname{cov} [\operatorname{Re} [\Phi_{n}^{\prime} (z)], \operatorname{Im} [\Phi_{n} (z) - v]]
 = -\operatorname{cov} [\operatorname{Im} [\Phi_{n}^{\prime} (z)], \operatorname{Re} [\Phi_{n} (z) - v]].
\end{equation*}
We compute the matrices $\boldsymbol{\Sigma}_{\mathbf{Z} \mathbf{Z}}$ and $\boldsymbol{\Sigma}_{\mathbf{Z} \overline{\mathbf{Z}}}$ as follows:

We have
$\eta_{j} \overline{\eta_{k}} = \alpha_{j} \alpha_{k} + \beta_{j} \beta_{k} - i (\alpha_{j} \beta_{k} - \beta_{j} \alpha_{k})$ for $j, k = 0, \dotsc, n$. Since $\alpha_{j}, \beta_{j} \sim \mathcal{N} (0, \frac{1}{2})$, we have the condition $\operatorname{\mathbb{E}} [\alpha_{j} \alpha_{k}] = \operatorname{\mathbb{E}} [\beta_{j} \beta_{k}] = \frac{1}{2}$ if $j = k$, but $\operatorname{\mathbb{E}} [\alpha_{j} \alpha_{k}] = \operatorname{\mathbb{E}} [\beta_{j} \beta_{k}] = 0$ if $j \neq k$. Notice that if $j = k$, then $\operatorname{\mathbb{E}} [\eta_{j} \overline{\eta_{k}}] = 1$. Straightforward calculation shows that
\begin{equation*}
\begin{split}
\operatorname{\mathbb{E}} [\Phi_{n}^{\prime} (z) \overline{\Phi_{n}^{\prime} (z)}]
 & = \operatorname{\mathbb{E}} \left[\sum_{j = 0}^{n} \sum_{k = 0}^{n} \eta_{j} \overline{\eta_{k}} \phi_{j}^{\prime} (z) \overline{\phi_{k}^{\prime} (z)}\right]
 = \sum_{j = 0}^{n} \phi_{j}^{\prime} (z) \overline{\phi_{j}^{\prime} (z)}, \\ \addlinespace[1mm]
\operatorname{\mathbb{E}} [\Phi_{n}^{\prime} (z) \overline{\Phi_{n} (z)}]
 & = \operatorname{\mathbb{E}} \left[\sum_{j = 0}^{n} \sum_{k = 0}^{n} \eta_{j} \overline{\eta_{k}} \phi_{j}^{\prime} (z) \overline{\phi_{k} (z)}\right]
 = \sum_{j = 0}^{n} \phi_{j}^{\prime} (z) \overline{\phi_{j} (z)}, \\ \addlinespace[1mm]
\operatorname{\mathbb{E}} [\Phi_{n} (z) \overline{\Phi_{n} (z)}]
 & = \operatorname{\mathbb{E}} \left[\sum_{j = 0}^{n} \sum_{k = 0}^{n} \eta_{j} \overline{\eta_{k}} \phi_{j} (z) \overline{\phi_{k} (z)}\right]
 = \sum_{j = 0}^{n} \phi_{j} (z) \overline{\phi_{j} (z)},
\end{split}
\end{equation*}
and
\begin{equation*}
\operatorname{\mathbb{E}} [\overline{\Phi_{n}^{\prime} (z)} \Phi_{n} (z)]
 = \operatorname{\mathbb{E}} [\overline{\Phi_{n}^{\prime} (z) \overline{\Phi_{n} (z)}}]
 = \sum_{j = 0}^{n} \overline{\phi_{j}^{\prime} (z)} \phi_{j} (z).
\end{equation*}
Altogether, we see that
\begin{equation*}
\boldsymbol{\Sigma}_{\mathbf{Z} \mathbf{Z}}
 = 
\begin{bmatrix*}
\displaystyle \sum_{j = 0}^{n} \phi_{j}^{\prime} (z) \overline{\phi_{j}^{\prime} (z)} & \displaystyle \sum_{j = 0}^{n} \phi_{j}^{\prime} (z) \overline{\phi_{j} (z)} \\ \addlinespace[1mm]
\displaystyle \sum_{j = 0}^{n} \overline{\phi_{j}^{\prime} (z)} \phi_{j} (z) & \displaystyle \sum_{j = 0}^{n} \phi_{j} (z) \overline{\phi_{j} (z)}
\end{bmatrix*}.
\end{equation*}

Next, we obtain
\begin{equation*}
\begin{split}
\boldsymbol{\Sigma}_{\mathbf{Z} \overline{\mathbf{Z}}}
 = \operatorname{\mathbb{E}} [(\mathbf{Z} - \boldsymbol{\upmu}_{\mathbf{Z}}) (\mathbf{Z} - \boldsymbol{\upmu}_{\mathbf{Z}})^{T}]
 & = \operatorname{\mathbb{E}} 
\begin{bmatrix*}
\Phi_{n}^{\prime} (z) \Phi_{n}^{\prime} (z) & \Phi_{n}^{\prime} (z) \Phi_{n} (z) \\ \addlinespace[1mm]
\Phi_{n}^{\prime} (z) \Phi_{n} (z) & \Phi_{n} (z) \Phi_{n} (z)
\end{bmatrix*}.
\end{split}
\end{equation*}
We treat $\boldsymbol{\Sigma}_{\mathbf{Z} \overline{\mathbf{Z}}}$ as we did $\boldsymbol{\Sigma}_{\mathbf{Z} \mathbf{Z}}$. We have
$\eta_{j} \eta_{k} = \alpha_{j} \alpha_{k} - \beta_{j} \beta_{k} + i (\alpha_{j} \beta_{k} + \beta_{j} \alpha_{k})$ for $j, k = 0, \dotsc, n$, and so $\operatorname{\mathbb{E}} [\eta_{j} \eta_{k}] = 0$ if $j = k$. Then
\begin{equation*}
\boldsymbol{\Sigma}_{\mathbf{Z} \overline{\mathbf{Z}}}
 = 
\begin{bmatrix*}
0 & 0 \\ \addlinespace[1mm]
0 & 0
\end{bmatrix*}
\end{equation*}
follows in a like manner.

By Theorem \ref{Thm-1.1}, from \eqref{eq-1.2} (with $\rho = 1$) and \eqref{eq-3.2}, we find
\begin{align*}
\Psi_{n, v} (z)
 & = \frac{\operatorname{\mathbb{E}} [\Phi_{n}^{\prime} (z)]}{\operatorname{\mathbb{E}} [\Phi_{n} (z) - v]} + \left(\frac{\operatorname{\mathbb{E}} [\Phi_{n}^{\prime} (z) \overline{\Phi_{n} (z)}]}{\operatorname{\mathbb{E}} [\Phi_{n} (z) \overline{\Phi_{n} (z)}]} - \frac{\operatorname{\mathbb{E}} [\Phi_{n}^{\prime} (z)]}{\operatorname{\mathbb{E}} [\Phi_{n} (z) - v]}\right) \exp \left(-\frac{\lvert \operatorname{\mathbb{E}} [\Phi_{n} (z) - v] \rvert^{2}}{\operatorname{\mathbb{E}} [\Phi_{n} (z) \overline{\Phi_{n} (z)}]}\right) \\
 & = \frac{\displaystyle \sum_{j = 0}^{n} \overline{\phi_{j} (z)} \phi_{j}^{\prime} (z)}{\displaystyle \sum_{j = 0}^{n} \phi_{j} (z) \overline{\phi_{j} (z)}} \exp \left[-\lvert v \rvert^{2} \left(\displaystyle \sum_{j = 0}^{n} \phi_{j} (z) \overline{\phi_{j} (z)}\right)^{-1}\right]. \numberthis \label{eq-3.7}
\end{align*}
Writing \eqref{eq-3.7} in terms of the kernels $K_{n}$ and $K_{n}^{(0, 1)}$, and then substituting in \eqref{eq-3.1}, Theorem \ref{Thm-1.2} can be thereby obtained.
\end{proof}

\section{Proof of Theorem \ref{Thm-1.3}} \label{section_4}

On the basis of the argument in the preceding section, it will be easy now to prove Theorem \ref{Thm-1.3}.

\begin{proof}
To enunciate the point that $\Psi_{n, v}$ is a function of $z$ and $\overline{z}$ we now write $\Psi_{n, v} (z, \overline{z})$. From Stokes's theorem in the plane, better known as Green's theorem, when expressed in terms of the partial derivative $\partial / \partial \overline{z}$, it is asserted that
\begin{equation*}
\oint_{\partial \Gamma} u (z, \overline{z}) \, dz
 = 2 i \iint_{\Gamma} \frac{\partial}{\partial \overline{z}} [u (z, \overline{z})] \, dx \, dy
\end{equation*}
for every smooth bounded open domain $\Gamma \subset \mathbb{C}$ with positively oriented boundary $\partial \Gamma$ and complex-valued function $u  (z, \overline{z}) \in C^{1} (\overline{\Gamma})$ (see equation (2.37) of \cite{Taylor2011}). In fact, if in the formula we change $u  (z, \overline{z})$ to $\Psi_{n, v}  (z, \overline{z})$ and choose $\Theta$ as $\Gamma$, there results
\begin{equation} \label{eq-4.1}
\operatorname{\mathbb{E}} [N_{\Theta}^{\Phi_{n}, v}]
 = \frac{1}{\pi} \iint_{\Theta} \frac{\partial}{\partial \overline{z}} [\Psi_{n, v} (z, \overline{z})] \, dx \, dy.
\end{equation}
A direct calculation using \eqref{eq-3.7} furnishes
\begin{equation*}
\begin{split}
\frac{\partial}{\partial \overline{z}} \left[\Psi_{n, v} (z, \overline{z})\right]
 & = \frac{1}{\displaystyle \sum_{j = 0}^{n} \phi_{j} (z) \overline{\phi_{j} (z)}} \exp \left[- \lvert v \rvert^{2} \left(\displaystyle \sum_{j = 0}^{n} \phi_{j} (z) \overline{\phi_{j} (z)}\right)^{-1}\right] \\ & \quad \times \left\{\sum_{j = 0}^{n} \phi_{j}^{\prime} (z) \overline{\phi_{j}^{\prime} (z)} - \frac{\displaystyle \left|\sum_{j = 0}^{n} \phi_{j} (z) \overline{\phi_{j}^{\prime} (z)}\right|^{2}}{\displaystyle \sum_{j = 0}^{n} \phi_{j} (z) \overline{\phi_{j} (z)}} \left[1 - \lvert v \rvert^{2} \left(\displaystyle \sum_{j = 0}^{n} \phi_{j} (z) \overline{\phi_{j} (z)}\right)^{-1}\right]\right\}.
\end{split}
\end{equation*}
Expressing this in terms of the kernels $K_{n}, K_{n}^{(0, 1)}$, and $K_{n}^{(1, 1)}$, and next substituting in \eqref{eq-4.1}, after regrouping the terms, with a little algebra we will have, then, produced Theorem \ref{Thm-1.3}.
\end{proof}

\section{Nevai class} \label{section_5}

Let $\{\varphi_{j}\}_{j = 0}^{\infty}$ be OPUC with respect to a probability Borel measure $\mu$ on $\mathbb{T}$ and satisfy the orthogonality condition
\begin{equation} \label{eq-5.1}
\int_{\mathbb{T}} \varphi_{j} (e^{i \theta}) \overline{\varphi_{k} (e^{i \theta})} \, d \mu (e^{i \theta})
 = \delta_{j k}, \quad j, k \in \mathbb{N} \cup \{0\},
\end{equation}
where $\delta_{j k}$ denotes the Kronecker delta. OPUC are a direct generalization of the monomials $\{z^{j}\}_{j = 0}^{\infty}$, which have $d \mu (\theta) = \frac{1}{2 \pi} d \theta$. If $\{\varphi_{j}\}_{j = 0}^{\infty}$ heeds the Nevai class \cite{Nevai1979} and is locally uniformly whenever $z \in \mathbb{D}$, then
\begin{equation} \label{eq-5.2}
\lim_{n \to \infty} \frac{\varphi_{n}(z)}{\varphi_{n}^{\ast} (z)}
 = 0,
\end{equation}
where $\varphi_{n}^{\ast} (z) = z^{n} \overline{\varphi_{n} (1 / \overline{z})}$. A useful deduction of Theorem \ref{Thm-1.3} is the asymptotic behavior of $\rho_{n, v}$ on $\mathbb{D}$. The analysis can be carried out using the Christoffel--Darboux kernel
\begin{equation*}
\sum_{j = 0}^{n} \varphi_{j} (z) \overline{\varphi_{j} (w)}
 = \frac{\overline{\varphi_{n + 1}^{\ast} (w	)} \, \varphi_{n + 1}^{\ast} (z) - \overline{\varphi_{n + 1} (w)} \, \varphi_{n+ 1} (z)}{1 - \overline{w}z},
\end{equation*}
whenever $z, w \in \mathbb{C}$ with $\overline{w} z \neq 1$ (see Theorem 2.2.7 of \cite{Simon2005}). The following asymptotic formula for $\rho_{n, v}$ proves to be well-defined, since all the zeros of $\varphi_{n + 1}^{\ast}$ occur in $\mathbb{C} \backslash \mathbb{D}$ (see Theorem 1.7.1 of \cite{Simon2005}):

\begin{theorem} \label{Thm-5.1}
Let the basis functions for $\Phi_{n}$ be OPUC $\{\varphi_{j}\}_{j = 0}^{\infty}$ obeying the Nevai class. Then locally uniformly whenever $z \in \mathbb{D}$,
\begin{equation*}
\begin{split}
& \rho_{n, v} (z) \\
 & \quad = \frac{1}{\pi} \left(\frac{1}{(1 - \lvert z \rvert^{2})^{2}} + \frac{\lvert v \rvert^{2} [1 - \lvert z \rvert^{2} + o (1)]}{\lvert \varphi_{n + 1}^{\ast} (z) \rvert^{2}} \left|\frac{z}{1 - \lvert z \rvert^{2}} + \frac{\overline{\varphi_{n + 1}^{\ast \, \prime} (z)} + o (1)}{\overline{\varphi_{n + 1}^{\ast} (z)}}\right|^{2}\right) \exp \left(-\frac{\lvert v \rvert^{2} (1 - \lvert z \rvert^{2})}{ \lvert \varphi_{n + 1}^{\ast} (z) \rvert^{2}} + o (1)\right).
\end{split}
\end{equation*}
\end{theorem}

\begin{proof}
From equations (55), (58), and (59) of \cite{CorleyLedoan2020}, we have the OPUC Christoffel--Darboux kernels
\begin{gather}
K_{n} (z, z)
 = \frac{\lvert \varphi_{n + 1}^{\ast} (z) \rvert^{2} - \lvert \varphi_{n + 1} (z) \rvert^{2}}{1 - \lvert z \rvert^{2}}, \label{eq-5.3} \\ \addlinespace[1mm]
K_{n}^{(0, 1)} (z, z)
 = \frac{\overline{\varphi_{n + 1}^{\ast \, \prime} (z)} \varphi_{n + 1}^{\ast} (z) - \overline{\varphi_{n+ 1}^{\prime} (z)} \varphi_{n + 1} (z)}{1 - \lvert z \rvert^{2}} + \frac{z K_{n} (z, z)}{1 - \lvert z \rvert^{2}}, \label{eq-5.4}
\end{gather}
and
\begin{equation} \label{eq-5.5}
K_{n}^{(1, 1)} (z, z)
 = \frac{\lvert \varphi_{n + 1}^{\ast \, \prime} (z) \rvert^{2} - \lvert \varphi_{n + 1}^{\prime} (z) \rvert^{2}}{1 - \lvert z \rvert^{2}} + \frac{2 \operatorname{Re} [\overline{z} K_{n}^{(0, 1)} (z, z)] + K_{n} (z, z)}{1 - \lvert z \rvert^{2}}.
\end{equation}
It follows from Theorem \ref{Thm-1.3}, in conjunction with \eqref{eq-5.2} and \eqref{eq-5.3}, that
\begin{equation} \label{eq-5.6}
\exp \left(-\frac{\lvert v \rvert^{2}}{K_{n} (z, z)}\right)
 = \exp \left(-\frac{\lvert v \rvert^{2} (1 - \lvert z \rvert^{2})}{\lvert \varphi_{n + 1}^{\ast} (z) \rvert^{2}} + o (1)\right).
\end{equation}
Locally uniformly whenever $z \in \mathbb{D}$, from \eqref{eq-5.2}
\begin{equation} \label{eq-5.7}
\lim_{n \to \infty} \left(\frac{\varphi_{n} (z)}{\varphi_{n}^{\ast} (z)}\right)^{\prime}
 = 0.
\end{equation}
Combining \eqref{eq-5.2}--\eqref{eq-5.5} and \eqref{eq-5.7} results in
\begin{equation} \label{eq-5.8}
\frac{K_{n}^{(1, 1)} (z, z)}{K_{n} (z, z)} - \frac{\lvert K_{n}^{(0, 1)} (z, z) \rvert^{2}}{K_{n} (z, z)^{2}}
 = \frac{1}{(1 - \lvert z \rvert^{2})^{2}}+ o (1).
\end{equation}
Now, using \eqref{eq-5.2}--\eqref{eq-5.4} and the fact that $\varphi_{n}$ and $\varphi_{n}^{\prime}$ each approaches 0 when $n \to \infty$, we obtain
\begin{equation} \label{eq-5.9}
\frac{\lvert K_{n}^{(0, 1)} (z, z) \rvert^{2}}{K_{n} (z, z)^{3}}
 = \frac{1 - \lvert z \rvert^{2} + o (1)}{\lvert \varphi_{n + 1}^{\ast} (z) \rvert^{2}} \left|\frac{z}{1 - \lvert z \rvert^{2}} + \frac{\overline{\varphi_{n + 1}^{\ast \, \prime} (z)} + o (1)}{\overline{\varphi_{n + 1}^{\ast} (z)}}\right|^{2}.
\end{equation}
The assertion is deducible by Theorem \ref{Thm-1.3} and from \eqref{eq-5.6}, \eqref{eq-5.8}, and \eqref{eq-5.9}.
\end{proof}

\section{Szeg\H{o} class} \label{section_6}

If $\{\varphi_{j}\}_{j = 0}^{\infty}$ be OPUC in the Szeg\H{o} class, then $\mu$ is absolutely continuous with respect to the arc length measure with the nonnegative weight function $W$ (Radon--Nykodym derivative), defined and measurable on $[-\pi, \pi]$, for which the integrals $\int_{-\pi}^{\pi} W (\theta) \, d\theta$ and $\int_{-\pi}^{\pi} \lvert \log W (\theta) \rvert \, d\theta$ both exist. The former integral is assumed to be positive.

By Theorem 12.1.1 of \cite{Szego1975}, locally uniformly whenever $\lvert z \rvert < 1$, $\lim_{n \to \infty} \varphi_{n + 1}^{\ast} (z) = D (z)^{-1}$, where
\begin{equation*}
D (\xi)
 = \exp \left[\frac{1}{4 \pi} \int_{-\pi}^{\pi} \log W (\theta) \left(\frac{1 + \xi e^{-i \theta}}{1 - \xi e^{-i \theta}}\right) d \theta\right]
\end{equation*}
is uniquely determined by $W$ and analytic and nonzero whenever $\lvert \xi \rvert < 1$. Observe that $D (0) > 0$. Thus, whenever $\lvert z \rvert < 1$, $\lim_{n \to \infty} \varphi_{n + 1}^{\ast \, \prime} (z) = -D^{\prime} (z) / D (z)^{2}$. We apply Theorem \ref{Thm-5.1}. The following consequence is therefore clear:

\begin{theorem} \label{Thm-6.1}
Let the basis functions for $\Phi_{n}$ be OPUC $\{\varphi_{j}\}_{j = 0}^{\infty}$ obeying the Szeg\H{o} class. Then locally uniformly whenever $\lvert z \rvert < 1$,
\begin{equation*}
\lim_{n \to \infty} \rho_{n, v} (z)
 = \frac{1}{\pi} \left(\frac{1}{(1 - \lvert z \rvert^{2})^{2}} + (1 - \lvert z \rvert^{2}) \lvert v D (z) \rvert^{2} \left|\frac{z}{1 - \lvert z \rvert^{2}} - \frac{\overline{D^{\prime} (z)}}{\overline{D (z)}}\right|^{2}\right) \exp [-(1 - \lvert z \rvert^{2}) \lvert v D (z) \rvert^{2}].
\end{equation*}
\end{theorem}

\begin{corollary} \label{Cor-6.2}
Under the hypotheses of Theorem \ref{Thm-6.1}, if $\varphi_{j} (z) = z^{j}$, then, for every open disk $D_{\varrho} \subset \mathbb{D}$,
\begin{equation*}
\lim_{n \to \infty} \operatorname{\mathbb{E}} [N_{D_{\varrho}}^{\Phi_{n}, v}]
 = \frac{\varrho^{2}}{1 - \varrho^{2}} \exp [-\lvert v \rvert^{2} (1 - \varrho^{2})].
\end{equation*}
Furthermore, whenever $\lvert z \rvert < 1$,
\begin{equation*}
\lim_{n \to \infty} \frac{\rho_{n, v} (z)}{\operatorname{\mathbb{E}} [N_{D_{\varrho}}^{\Phi_{n}, v}]}
 = \frac{1 - \varrho^{2}}{\pi \varrho^{2}} \left(\frac{1}{(1 - \lvert z \rvert^{2})^{2}} + \frac{\lvert v z \rvert^{2}}{1 - \lvert z \rvert^{2}}\right).
\end{equation*}
\end{corollary}

\begin{proof}
For if $\varphi_{j} (z) = z^{j}$, then $D (z) = 1$ and Theorem 4 of \cite{CorleyLedoan2020} is recovered. Then locally uniformly whenever $\lvert z \rvert < 1$,
\begin{equation*}
\lim_{n \to \infty} \rho_{n, v} (z)
 = \frac{1}{\pi} \left(\frac{1}{(1 - \lvert z \rvert^{2})^{2}} + \frac{\lvert v z \rvert^{2}}{1 - \lvert z \rvert^{2}}\right) \exp [-\lvert v \rvert^{2} (1 - \lvert z \rvert^{2})].
\end{equation*}
It is readily seen that
\begin{equation*}
\lim_{n \to \infty} \operatorname{\mathbb{E}} [N_{D_{\varrho}}^{\Phi_{n}, v}]
 = 2 \int_{0}^{\varrho} \left(\frac{1}{(1 - r^{2})^{2}} + \frac{\lvert v \rvert^{2} r^{2}}{1 - r^{2}}\right) \exp [-\lvert v \rvert^{2} (1 - r^{2})] r \, dr.
\end{equation*}
Express the integrand as a derivative with respect to the radial coordinate $r$ and apply the first fundamental theorem of calculus.
\end{proof}

\section{Sthal, Totik, and Ullman class} \label{section_7}

Assume that $\mu$ is regular in the sense of Sthal, Totik, and Ullman \cite{StahlTotik1992, Ullman1972}. The leading coefficient $\kappa_{j}$ of $\varphi_{j}$ is subject to the condition
\begin{equation} \label{eq-7.1}
\lim_{j \to \infty} \sqrt[j]{\kappa_{j}}
 = 1.
\end{equation}

\begin{theorem} \label{Thm-7.1}
Let $\mu$ defined by \eqref{eq-5.1} for OPUC $\{\varphi_{j}\}_{j = 0}^{\infty}$ be a strictly positive Borel measure on $[-\pi, \pi[$, absolutely continuous with respect to the Lebesgue measure, and regular in the sense of Sthal, Totik, and Ullman. Assume that $\mu$ has a positive weight function that is continuous on $\mathbb{T}$. Then
\begin{equation*}
\lim_{n \to \infty} \operatorname{\mathbb{E}} \left[\frac{1}{n} \, N_{\mathbb{D}}^{\Phi_{n}, v}\right]
 = \frac{1}{2}.
\end{equation*}
\end{theorem}

\begin{proof}
In Theorem 3.1 of \cite{LevinLubinsky2007} it is proved that, uniformly whenever $\theta$ is in a compact set $J \subset [-\pi, \pi[$, with $z = e^{i \theta}$, there exist positive absolute constants $C_{1}$ and $C_{2}$ such that $C_{1} \leq \lim_{n \to \infty} \frac{1}{n} \, K_{n} (z, z) \leq C_{2}$. We infer from Corollary 1.2 of \cite{LevinLubinsky2007} that
\begin{equation*}
\lim_{n \to \infty} \frac{1}{n} \, \frac{\overline{K_{n}^{(0, 1)} (z, z)}}{K_{n} (z, z)}
 = \frac{1}{2} \, \overline{z}.
\end{equation*}
Thus
\begin{equation*}
\lim_{n \to \infty} \exp \left(-\frac{\lvert v \rvert^{2}}{K_{n} (z, z)}\right)
 = 1.
\end{equation*}
Since $\rho_{n, v}$ has zero mass on the real line, by Green's theorem
\begin{equation*}
\lim_{n \to \infty} \operatorname{\mathbb{E}} \left[\frac{1}{n} \, N_{\mathbb{D}}^{\Phi_{n}, v}\right]
 = \frac{1}{4 \pi i} \oint_{\mathbb{T}} \overline{z} \, dz
 = \frac{1}{2 \pi} \iint_{\mathbb{D}} dx \, dy
 = \frac{1}{2}.
\end{equation*}
\end{proof}

Theorem \ref{Thm-7.1} holds under the relaxed assumption that $\mu$ has a continuous weight function on $\mathbb{T} \backslash E$ for every countable set $E$, if required.

\section{Bergman polynomials} \label{section_8}

Let $\{p_{j}\}_{j = 0}^{\infty}$ be Bergman polynomials $p_{j} (z) = \kappa_{j} z^{j} + \dotsb$ with $\kappa_{j} > 0$ for $j = 0, 1, \dotso$, satisfying \eqref{eq-7.1} and
\begin{equation} \label{eq-8.1}
\int_{\mathbb{D}} p_{j} (z) \overline{p_{k} (z)} \, d \mu (z)
 = \left\{ \begin{array}{ll}
1 & \mbox{if $j = k$,} \\
0 & \mbox{if $j \neq k$.}
\end{array}
\right.
\end{equation}
Assume that $\mu$ is absolutely continuous with respect to the planar Lebesgue measure $d \mu (z) = w (x, y) dx dy$, where $w$ is the weight function.

\begin{theorem} \label{Thm-8.1}
Let the basis functions for $\Phi_{n}$ be Bergman polynomials $\{p_{j}\}_{j = 0}^{\infty}$ on $\mathbb{D}$. Then locally uniformly for every $z \in \mathbb{D}$,
\begin{equation*}
\lim_{n \to \infty} \rho_{n, v} (z)
 = \frac{2}{\pi} \left(\frac{1}{(1 - \lvert z \rvert^{2})^{2}} + 2 \pi \lvert v z \rvert^{2}\right) \exp [-\pi \lvert v \rvert^{2} (1 - \lvert z \rvert^{2})^{2}].
\end{equation*}
\end{theorem}

\begin{proof}
Pass the formula for $\rho_{n, v}$ of Theorem \ref{Thm-1.3} to the limit as $n \to \infty$. Apply $\lim_{n \to \infty} K_{n} (z, z) = \frac{1}{\pi} (1 - \lvert z \rvert^{2})^{-2}$ to obtain the limiting values of $K_{n}^{(0, 1)}$ and $K_{n}^{(1, 1)}$.
\end{proof}

\begin{corollary} \label{Cor-8.2}
Under the hypothesis of Theorem \ref{Thm-8.1}, for every open disk $D_{\varrho} \subset \mathbb{D}$ with $\varrho < 1$,
\begin{equation*}
\lim_{n \to \infty} \operatorname{\mathbb{E}} [N_{D_{\varrho}}^{\Phi_{n}, v}]
 = \frac{2 \varrho^{2}}{1 - \varrho^{2}} \exp [-\pi \lvert v \rvert^{2} (1 - \lvert z \rvert^{2})^{2}].
\end{equation*}
Furthermore, whenever $\lvert z \rvert < 1$,
\begin{equation*}
\lim_{n \to \infty} \frac{\rho_{n, v} (z)}{\operatorname{\mathbb{E}} [N_{D_{\varrho}}^{\Phi_{n}, v}]}
 = \frac{1 - \varrho^{2}}{\pi \varrho^{2}} \left(\frac{1}{(1 - \lvert z \rvert^{2})^{2}} + 2 \pi \lvert v z \rvert^{2}\right).
\end{equation*}
\end{corollary}

\begin{proof}
Appeal to integration in polar coordinates.
\end{proof}

\begin{theorem} \label{Thm-8.3}
Let $\mu$ defined by \eqref{eq-8.1} for Bergman polynomials $\{p_{j}\}_{j = 0}^{\infty}$ be a strictly positive Borel measure on $\mathbb{D}$, absolutely continuous with respect to the Lebesgue measure, and regular in the sense of Sthal, Totik, and Ullman. Assume that $\mu$ has a positive weight function that is continuous on $\mathbb{T}$. Then
\begin{equation*}
\lim_{n \to \infty} \operatorname{\mathbb{E}} \left[\frac{1}{n} \, N_{\mathbb{D}}^{\Phi_{n}, v}\right]
 = \frac{2}{3}.
\end{equation*}
\end{theorem}

\begin{proof}
Theorem 3.1 of \cite{Lubinsky2010} asserts that whenever $\theta \in J$ with $z = e^{i \theta}$ there exist positive absolute constants $C_{3}$ and $C_{4}$ such that $C_{3}  \leq \lim_{n \to \infty} \frac{1}{n^{2}} \, K_{n} (z, z) \leq C_{4}$. By Corollary 1.4 of \cite{Lubinsky2010},
\begin{equation*}
\lim_{n \to \infty} \frac{1}{n} \, \frac{\overline{K_{n}^{(0, 1)} (z, z)}}{K_{n} (z, z)}
 = \frac{2}{3} \, \overline{z}.
\end{equation*}
These with Green's theorem give the assertion.
\end{proof}

\section{Numerical simulation} \label{section_9}

In this final section we show a numerical simulation of a random polynomial $\Phi_{n}$ spanned by Bergman polynomials $p_{j}$ at various $v$-level crossings. For a fixed $k > 0$, we consider $p_{j} (z) = \sqrt{(j + 1) (j + k + 1) / (\pi k)} z^{j}$ for $j = 0, \dotsc, n$ with the weight function $w (z) = 1 - \lvert z \rvert^{2 k}$. We set $k = 2$. It is easy to check that the $p_{j}$ satisfy the conditions \eqref{eq-7.1} and \eqref{eq-8.1}.

It is interesting to take a closer look at the relationship between the mean $\operatorname{\mathbb{E}} [N_{\Theta}^{\Phi_{n}, v}]$ and the $v$-level crossings. Figure \ref{figure 1} shows the numerical values of $\operatorname{\mathbb{E}} [N_{\Theta}^{\Phi_{n}, v}]$ represented by the small dots when $n$ is small. The plot is produced with $v = 0 + i 0$, $10 + i 10$, $\dotsc$, $200 + i 200$ in the order left-to-right. The first few values of $\operatorname{\mathbb{E}} [N_{\Theta}^{\Phi_{n}, v}]$ about $v = 0 + i 0$ are $0.727273$, $1.46154$, $2.2$, $2.94118$, $3.68421$, $4.42857$, $5.17391$, $\dotso$. As $\lvert v \rvert$ increases with $n$, we may visualize the relationship between $\operatorname{\mathbb{E}} [N_{\Theta}^{\Phi_{n}, v}]$ and $v$ by plotting the points $(n, \operatorname{\mathbb{E}} [N_{\Theta}^{\Phi_{n}, v}])$ for each $v$ and connecting those points with a curved line. Initially, these lines shift to the right rather than upward. This is due to the fewer number of roots being counted. On close scrutiny it turns out that all the lines are moving upwards towards the leftmost one. Some appear to coincide with it. The differences between the numerical values for $v = 0 + i 0$ and the other levels seem to decrease to 0 when $\lvert v \vert$ increases with $n$ (see Corollaries \ref{Cor-1.4} and \ref{Cor-8.2}).

\begin{figure}[H]
\centering$
\begin{array}{cc}
\includegraphics[width=127mm]{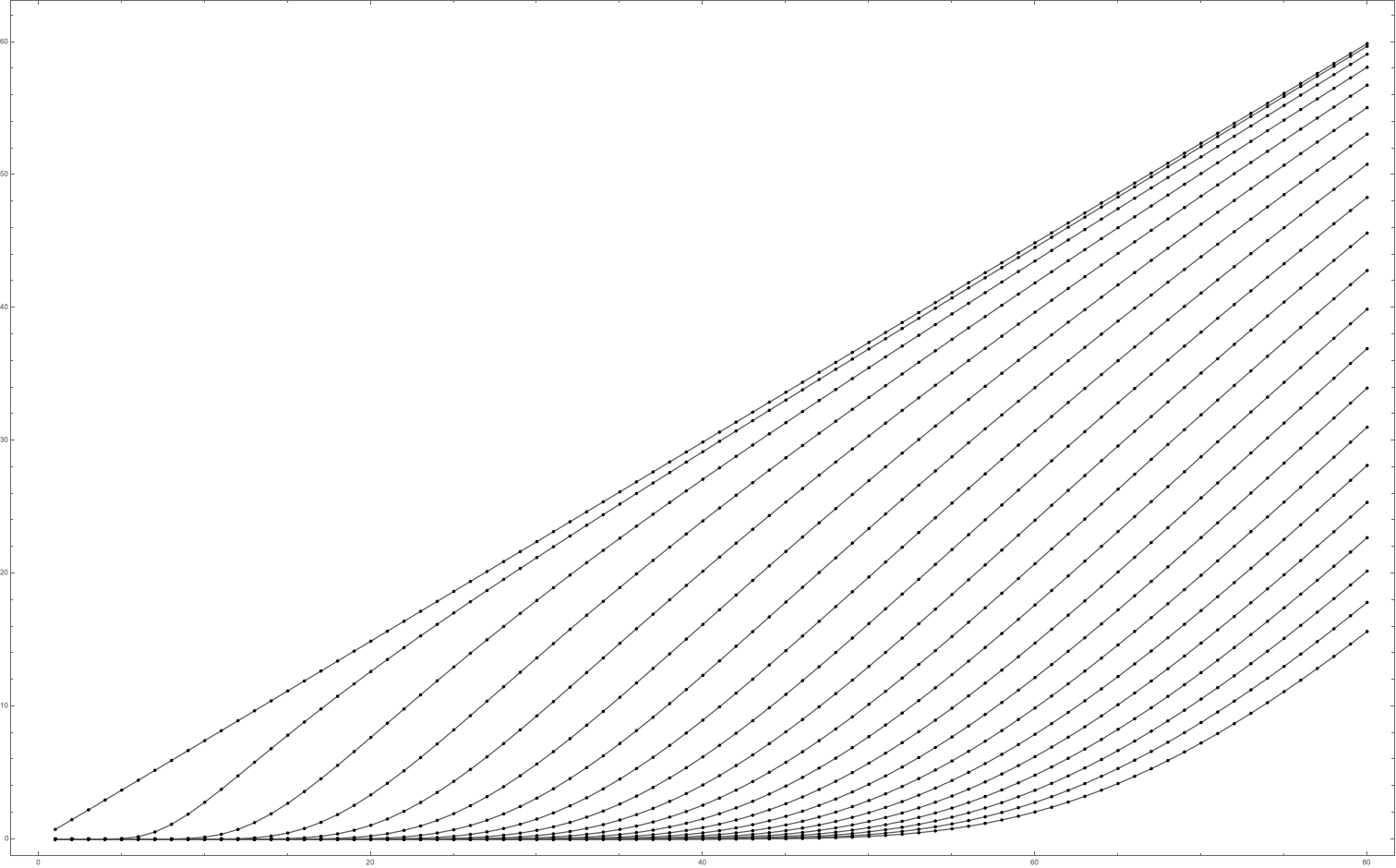}
\end{array}$
\caption{The effect of selection of different $v$-level crossings on the mean $\operatorname{\mathbb{E}} [N_{\Theta}^{\Phi_{n}, v}]$.}
\label{figure 1}
\end{figure}

It is informative to also study the profiles of the density $\rho_{n, v}$ and the roots of the random equation $\Phi_{n} (z) = v$ for some values of the parameters $n$ and $v$ and to compare the analytical results with the numerical evidence. In Figures \ref{figure 2} through \ref{figure 4} we plotted the density profiles obtained by a direct numerical search for the roots of $\Phi_{n} (z) = v$ with $n = 10$ and $v = 0 + i 0$, $5 + i 5$, $40 + i 40$. These results can be compared with the explicit formula for $\rho_{n, v}$ of Theorem \ref{Thm-1.3} for the same values of $n$ and $v$ and with the asymptotic profile of Theorem \ref{Thm-8.1}. In Figure \ref{figure 2} the left-hand plot is a picture of $\rho_{n, v}$ with $n = 10$ and $v = 0 + i 0$. The right-hand plot shows the empirical manifestation of 30,000 complex roots from randomly generated decics in $\Phi_{n} (z) = v$. The pictures in these figures show that the analytical results agree very well with the numerical data set. If we fix $n$ and increase $\lvert v \rvert$, then Figures \ref{figure 3} and \ref{figure 4} give a clear indication that the roots begin to move further away from $\mathbb{T}$ (see Corollary \ref{Cor-1.4}). Finally, if we fix $v$ and increase $n$, then the roots tend to concentrate near $\mathbb{T}$.

\begin{figure}[H]
\centering$
\begin{array}{cc}
\includegraphics[width=60mm]{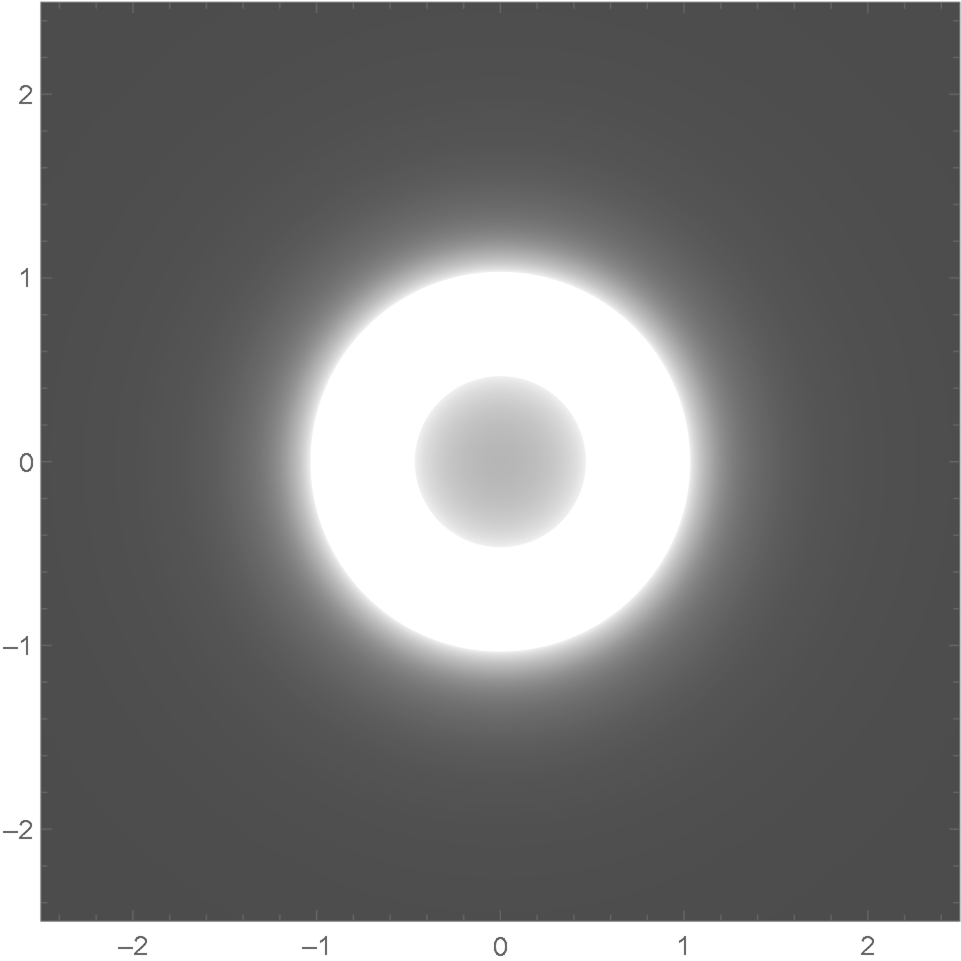} \quad \quad
\includegraphics[width=60mm]{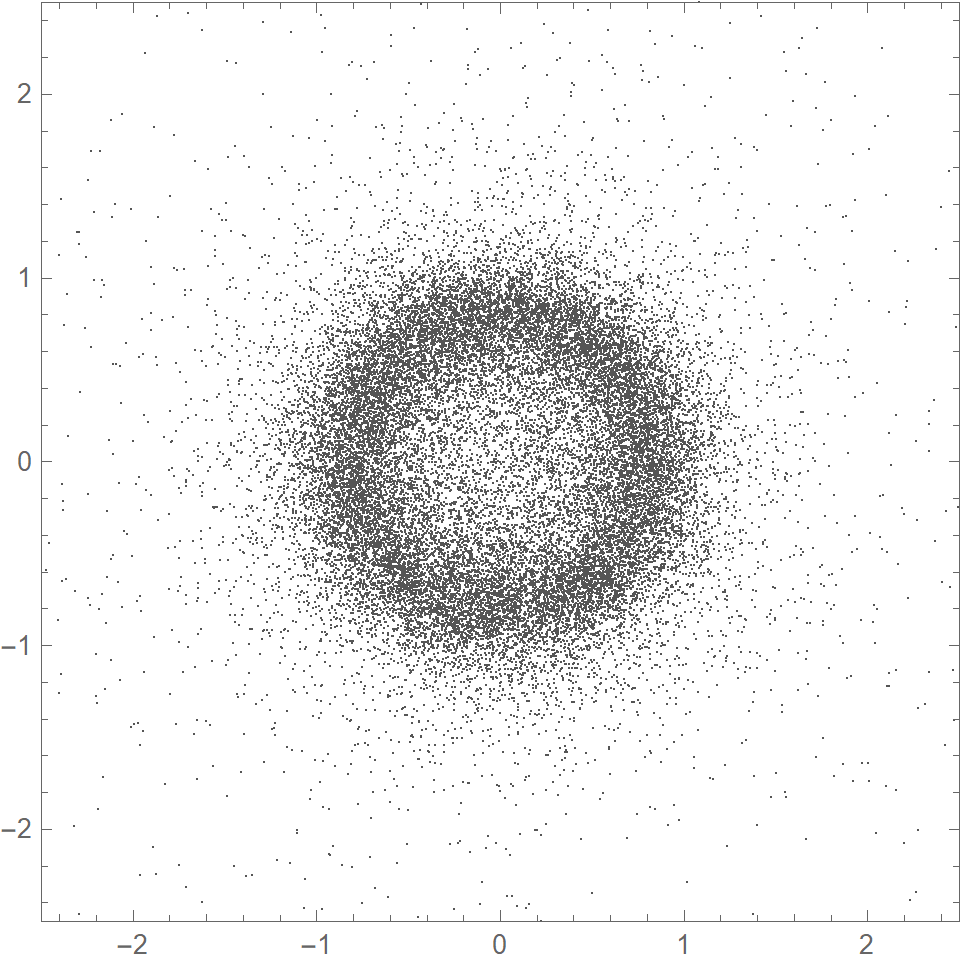}
\end{array}$
\caption{Analytical versus numerical density profiles for the zeros of random decics ($n = 10$) with independent and identically distributed standard complex normal coefficients and crossing about the level $v = 0 + i 0$.}
\label{figure 2}
\end{figure}

\begin{figure}[H]
\centering$
\begin{array}{cc}
\includegraphics[width=60mm]{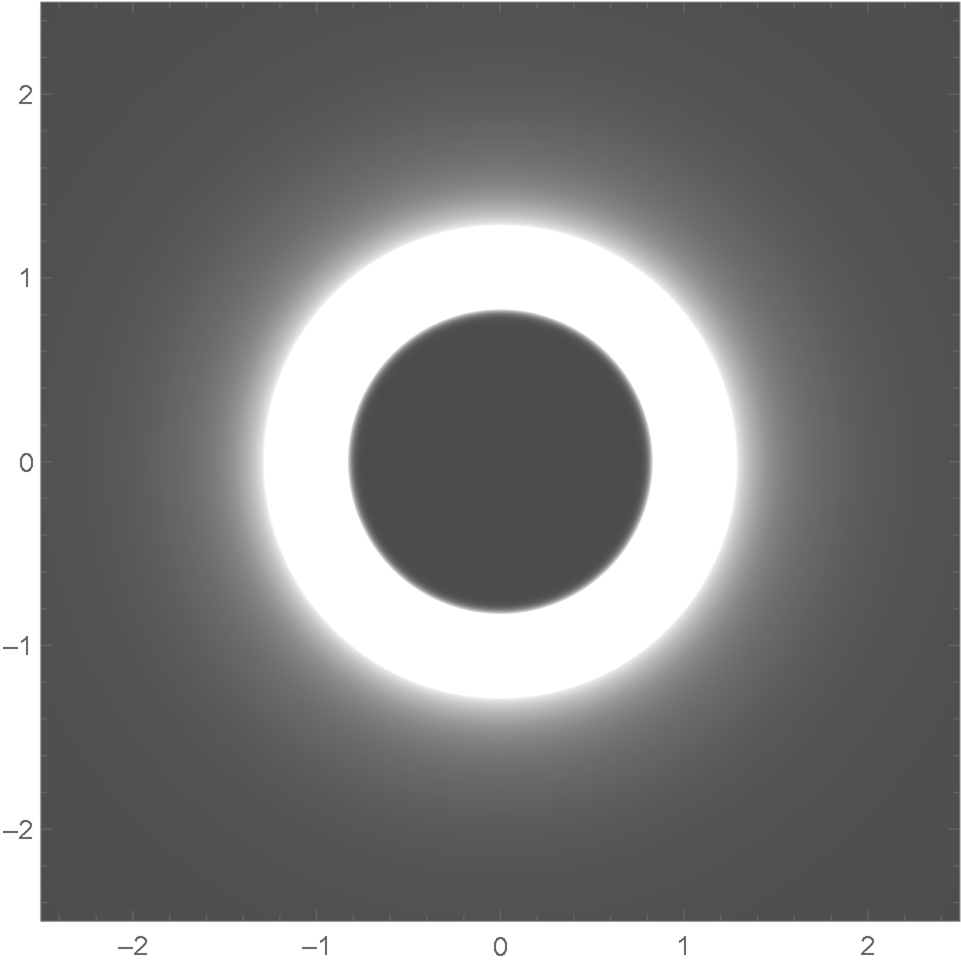} \quad \quad
\includegraphics[width=60mm]{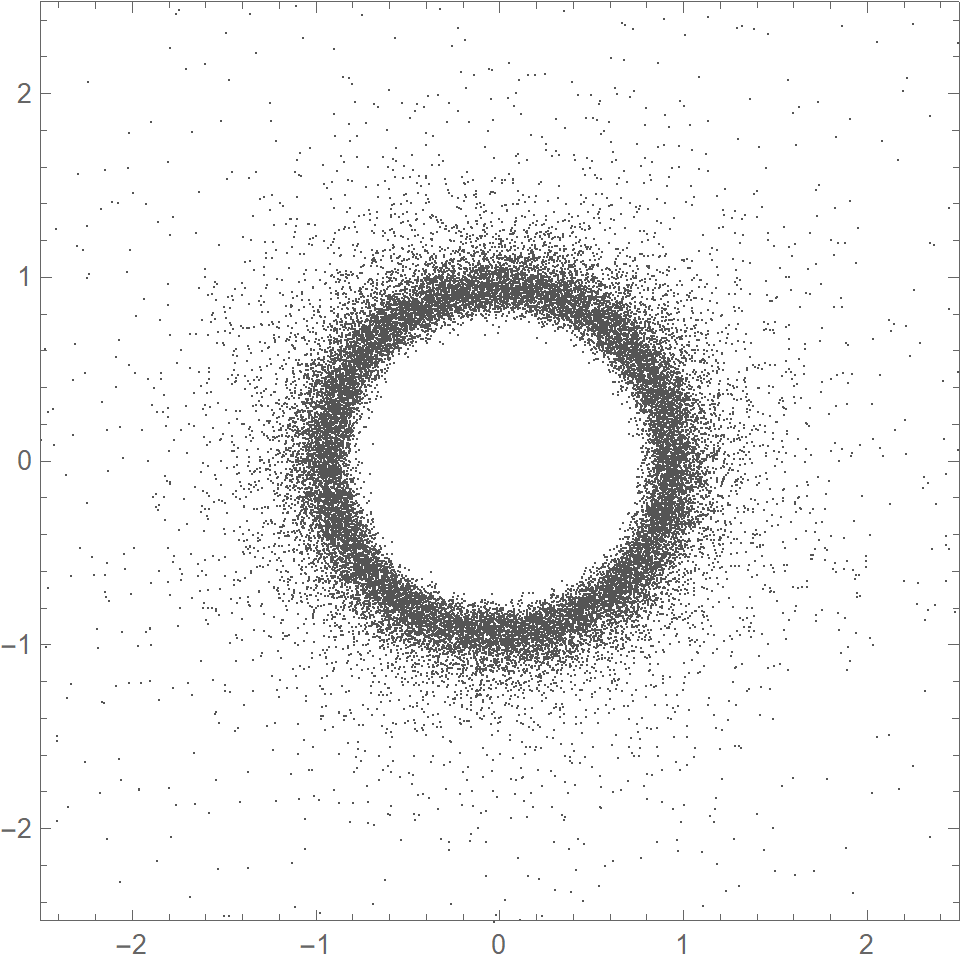}
\end{array}$
\caption{Density profiles for  $n = 10$ and $v = 5 + i 5$.}
\label{figure 3}
\end{figure}

\begin{figure}[H]
\centering$
\begin{array}{cc}
\includegraphics[width=60mm]{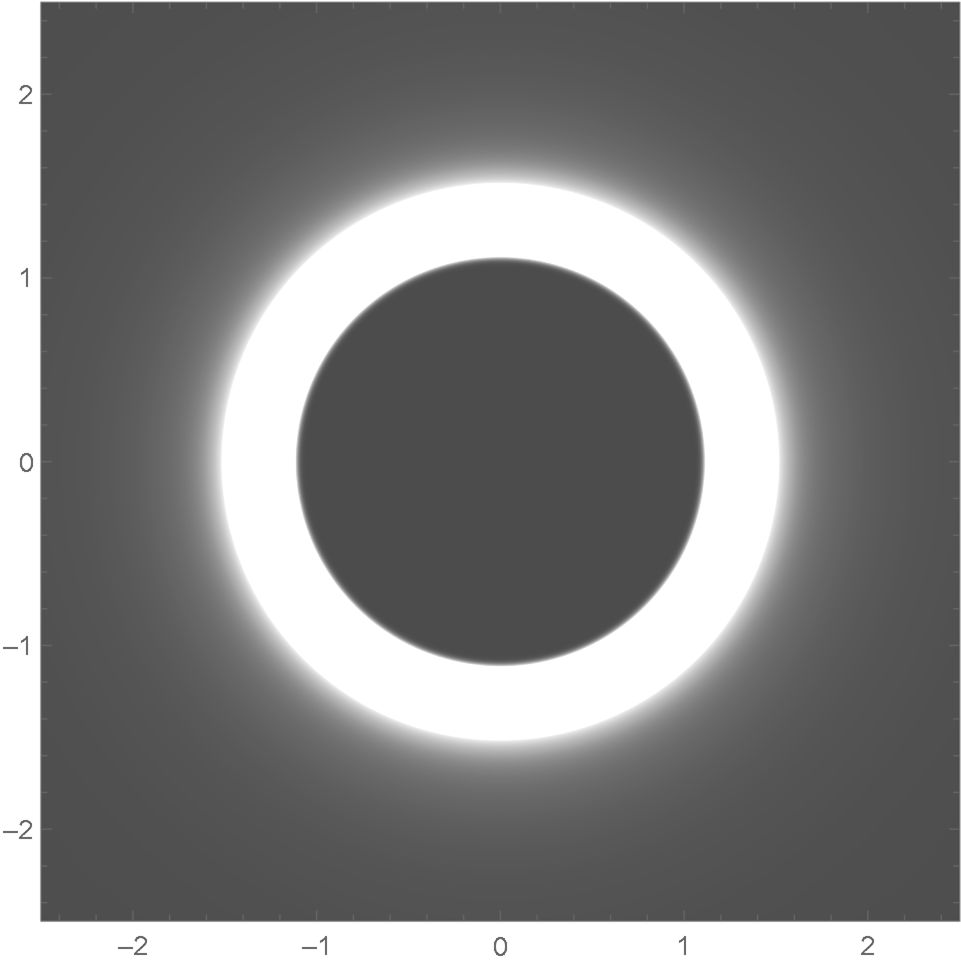} \quad \quad
\includegraphics[width=60mm]{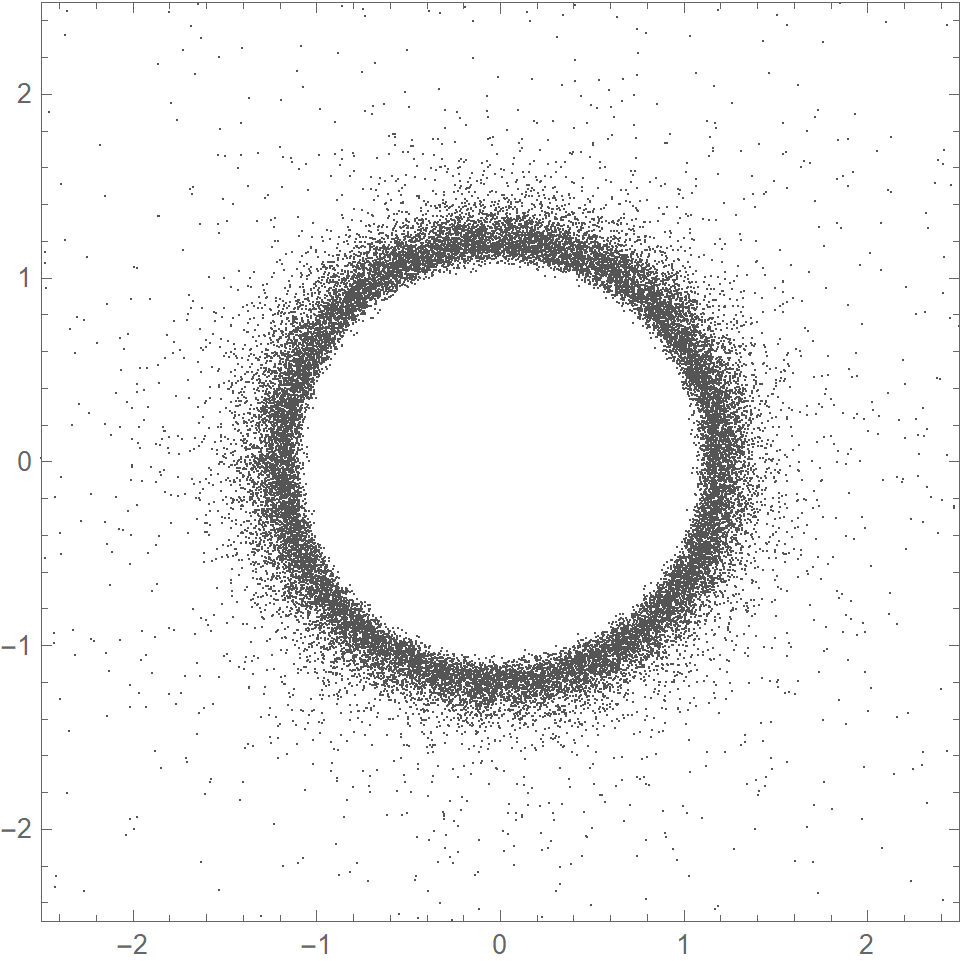}
\end{array}$
\caption{Density profiles for $n = 10$ and $v = 40 + i 40$.}
\label{figure 4}
\end{figure}

\end{document}